\documentclass[11pt,draft]{amsart}
\usepackage{dsfont}
\usepackage{amsfonts}
\usepackage{mathrsfs}
\usepackage{amssymb}
\usepackage{amsmath}
\usepackage{amsfonts}
\usepackage{amsfonts,enumerate}
\usepackage{pifont}
\usepackage{enumerate}
\usepackage{graphics}
\usepackage{verbatim}
\usepackage{amsthm}
\usepackage{latexsym,bm}
\usepackage{color}
\numberwithin{equation}{section}
\newtheorem{theorem}{Theorem}[section]
\newtheorem{corollary}[theorem]{Corollary}
\newtheorem{lemma}[theorem]{Lemma}
\newtheorem{proposition}[theorem]{Proposition}

\theoremstyle{definition}
\newtheorem{remark}[theorem]{Remark}

\newcommand{\R}{{\mathbb R}}
\newcommand{\Z}{\mathbb Z}
\newcommand{\B}{{\mathcal B}}
\newcommand{\E}{{\mathcal E}}
\newcommand{\A}{{\mathcal A}}
\newcommand{\V}{{\mathcal V}}
\newcommand{\Su}{\mathbb{S}}

\newcommand{\Q}{{\mathcal Q}}

\newcommand{\un}{{\mathds {1}}}

\begin{document}

\title[]{Vector valued $q$-variation for differential operators and semigroups I}

\author{Guixiang Hong$^*$}
\address{Instituto de Ciencias Matem\'aticas,
CSIC-UAM-UC3M-UCM, Consejo Superior de Investigaciones
Cient\'ificas, C/Nicol\'as Cabrera 13-15. \newline 28049, Madrid. Spain.\\
\emph{E-mail address: guixiang.hong@icmat.es}}
\thanks{*\ Corresponding author}

\author{Tao Ma}
\address{School of Mathematics and Statistics, Wuhan University, Wuhan 430072, China \\ \emph{E-mail address: tma.math@whu.edu.cn}}

\thanks{\small {{\it MR(2000) Subject Classification}.} Primary
42B25, 47B38, 47A35, 47D07; Secondary 46E40, 46B20.}

\maketitle

\begin{abstract}
In this paper, we establish $\mathcal B$-valued variational inequalities for differential operators, ergodic averages and symmetric diffusion semigroups under the condition that Banach space $\mathcal B$ has martingale cotype property. These results generalize, on the one hand Pisier and Xu's result on the  variational inequalities for $\mathcal B$-valued martingales, on the other hand many classical variational inequalities in harmonic analysis and ergodic theory. Moreover, we show that Rademacher cotype $q$ is necessary for the $\mathcal B$-valued  $q$-variational inequalities. As applications of the variational inequalities, we deduce the jump estimates and  obtain quantitative information on the rate of convergence. It turns out the rate of convergence depends on the geometric property of the Banach space under consideration, which considerably improve Cowling and Leinert's result where it is shown that the convergence always holds for all Banach spaces.
\end{abstract}

\section{Introduction}
Let $(\Omega,\mu)$ be a measure space and $\B$ be a Banach space. A submarkovian $C_0$ semigroup $(T_t)_{t\geq0}$ acting on $L^p(\Omega)$ extends to a semigroup of operators on $L^p(\Omega;\B)$. Cowling and Leinert showed in \cite{CoLe11} that
\begin{align}\label{motivation1}
\|T_tf(\omega)-f(\omega)\|\rightarrow 0,\;\mathrm{a.e.}\;\omega\in\Omega,\;\mathrm{as}\;t\rightarrow0^+
\end{align}
for any Banach space $\B$ and any $f\in L^p(\Omega;\B)$ with $1<p<\infty$, where (and throughout the paper) $\|\cdot\|$ means taking $\B$-norm. The predecessor of this result is the one \cite{Tag09} by Taggart where the convergence was shown only for Banach spaces having UMD property.

In this paper, we try to understand more precise information on this convergence (\ref{motivation1}), which leads us to consider the vector-valued variational inequalities since it is well-known that variational inequalities can be used to measure the speed of convergence for the family of operators in consideration.

Actually, the scalar-valued variational inequalities have been studied a lot in probability, ergodic theory and harmonica analysis. The first variational inequality was proved by L\'epingle \cite{Lep76} for martingales which improves the classical Doob maximal inequality. Thirteen years later, Bourgain in \cite{Bou89} proved the variational inequality for the ergodic averages of a dynamic system, which has inaugurated a new research direction in ergodic theory and harmonic analysis.  Bourgain's work was considerably improved and extended to many other operators in ergodic theory. For instance, Jone et al in \cite{JKRW98} and \cite{JRW03}  established variational inequalities for differential operators and ergodic averages of measure-preserving invertible transforms; Le Merdy and Xu \cite{LeXu2} obtained very recently variational inequalities for the contractively regular semigroups under an analyticity assumption. On the other hand, almost in the same period, variational inequalities
  have been studied in harmonic analysis too. The first works on this subject are \cite{CJRW00} and \cite{CJRW03} where Campbell et al proved the variational inequalities for singular integrals. Since then variational inequalities for different kinds of operators in harmonic analysis have been built; see e.g. \cite{DMT12} for paraproducts, \cite{JSW08} for differential operator and Hilbert transform along curves,  \cite{JoWa04} for the Fej\'er and Poisson kernels, \cite{Mas},\cite{MaTo12} and \cite{MaTo} for singular integrals on Lipschitz graphs, as well as \cite{OSTTW12} for Carleson operator.

However, regarding the vector-valued variational inequalities, there is only one result as we know. That is, the $q$-variational inequality for vector-valued martingales established by Pisier and Xu in \cite{PiXu88}, where actually the authors reproved L\'epingale's scalar-valued variational inequality using another approach which can be easily adapted to the vector-valued setting. To state their inequality we need to recall the definition of the vector-valued $q$-variation. Give a sequence $(a_n)_{n\geq0}$ in Banach space $\B$ and a number $1\leq q<\infty$, the vector-valued $q$-variation norm is defined as
$$\|(a_n)_{n\geq0}\|_{v_q(\B)}=\sup\{(\|a_{n_0}\|^q+\sum_{k\geq1}\|a_{n_{k}}-a_{n_{k-1}}\|^{q})^{\frac{1}{q}}\}$$
where the supremum runs over all increasing sequences $(n_k)_{k\geq0}$ of integers. It is clear that the set of $v_q(\B)$ of all sequences with a finite vector-valued $q$-variation is a Banach space with respect to the norm $v_q(\B)$.

Let $(\mathcal{E}_n)_{n\geq0}$ is an increasing sequence of conditional expectations on a probability space $\Omega$. Pisier and Xu proved that if $\B$ is of martingale cotype $q_0$ with $2\leq q_0<\infty$,
\begin{align}\label{martingale cotype q0}
\|\mathcal{E}_0f\|+\sum_{n\geq1}\mathcal{E}(\|\mathcal{E}_nf-\mathcal{E}_{n-1}f\|^{q_0})\leq C_{q_0}\mathcal{E}(\|f\|^{q_0}),
\end{align}
then
\begin{align}\label{pisierxu}
\|(\mathcal{E}_nf)_{n\geq0}\|_{L^p(v_q(\B))}\leq C_{p,q}\|\|f\|\|_p,\;\forall f\in L^p(\Omega;\B)
\end{align}
provided $q_0<q<\infty$ and $1<p<\infty$.

The main result of this paper says that similar inequalities as (\ref{pisierxu}) remain true in harmonic analysis for differential operators and in ergodic theory for symmetric diffusion semigroups  instead of conditional expectations. To state the main results, we need to recall more notations. A symmetric diffusion semigroup  is a bounded strongly continous semigroup $(T_t)_{t\geq0}$ defined simutaneously on $L^p(\Omega,\mu)$, $1\leq p\leq\infty$ such that for all $t>0$,
\begin{enumerate}[(i)]
\item $\|T_t\|_{L^p\rightarrow L^p}\leq1$, $\forall 1\leq p\leq\infty$;
\item $\langle T_tf,g\rangle$=$\langle f,T_tg\rangle$ whenever $f,g\in L^2(\Omega)$;
\item $T_tf\geq0$ whenever $f\geq0$;
\item $T_t1=1$.
\end{enumerate}
We also need  the following continuous analog of $v_q(\B)$. Given a family $(a_t)_{t>0}$ in $\B$, define
$$\|(a_t)_{t>0}\|_{V_q(\B)}=\sup\{(\|a_{t_0}\|^q+\sum_{k\geq1}\|a_{t_{k}}-a_{t_{k-1}}\|^{q})^{\frac{1}{q}}\}$$
where the supremum runs over all increasing sequences $(t_k)_{k\geq0}$ of positive real numbers. Then we define $V_q(\B)$ to be the Banach space of all $(a_t)_{t>0}$ with $V_q(\B)$-norm finite.

The first main result of this paper is stated as follows.  Let $(T_t)_{t\geq0}$ be a symmetric diffusion semigroup. Let $2\leq q_0<\infty$ and $\B$ be a Banach space of martingale cotype $q_0$ which is an interpolation space between a Hilbert space and another Banach space $B_0$ of martingale cotype $2\leq q_1<\infty$. Then for any $q_0< q<\infty$ and any $f\in L^p(\Omega;\B)$,  the family $(T_t(f))_{t>0}$ belongs to $V_q(\B)$ for a.e. $\omega\in\Omega$ and we have an estimate
\begin{align}\label{analytic semigroup continuous0}
 \left\|\|\omega\mapsto (T_t(f)(\omega))_{t>0}\|_{V_q(\B)}\right\|_{p}\leq C_{p,q} \|\|f\|\|_p,\;\forall f\in L^p(\Omega;\B).
\end{align}

This result can be viewed as a vector-valued version of Corollary 4.5 in \cite{LeXu2}. Restricted to the symmetric diffusion semigroups and the Banach spaces satisfying the above conditions, this result provides quantitative information of the convergence.  Interestingly, it is shown in Section 6 that  the speed of the convergence depends on the quantity of martingale cotype of the Banach space under consideration.

A priori, the vector-valued inequality is difficult to deal with, as in the vector-valued harmonica analysis. However, the arguments used in the scalar valued case \cite{LeXu2} is very powerful, where pointwise estimates are used, so that the general pattern can be adapted to the vector-valued case. More precisely, as in \cite{LeXu2} we will deduce inequality (\ref{analytic semigroup continuous0}) from a similar estimate for discrete semigroups using an approxiamtion argument based on the semigroup property, which in turn relys on a similar estimate for ergodic averages $M(T_t)$'s (see Section 3 for the definition)
\begin{align}\label{ergodic average0}
\|(M_n(T_t)f)_{n\geq0}\|_{L^p(v_q(\B))}\leq C_{p,q} \|\|f\|\|_p,\;\forall f\in L^p(\Omega;\B),
\end{align}
provided $q_0<q<\infty$,
and  vector-valued Littlewood-Paley inequality
\begin{align}\label{littlewood-palye0}
\left\|\left(\sum^{\infty}_{n=0}\frac{1}{n+1}\|(n+1)T^n_t(T_t-I)f\|^{q_0}\right)^{\frac{1}{q_0}}\right\|_{p}\leq C_{p,q_0}\|\|f\|\|_p.
\end{align}

To obtain (\ref{littlewood-palye0}), we use complex interpolation and the fact that every $T_t$ admits Rota's dilation. The reason why we can not obtain inequality (\ref{analytic semigroup continuous0}) for all Banach spaces of martingale cotype $q_0$ is that we have not yet been able to prove (\ref{littlewood-palye0}) for all Banach spaces of martingale cotype $q_0$, which remains an open problem. On the other hand, inequality (\ref{littlewood-palye0}) is sharp in the sense that the inequality holds true for one $1<p<\infty$ (or equivalent all $1<p<\infty$) implies that $\B$ has to be of martingale cotype $q_0$. See for instance \cite{MTX06}.

On the estimate (\ref{ergodic average0}), we prove a little bit stronger result, a similar estimate for all Banach spaces of martingale cotype $q_0$ and all continuous contractively regular semigroups,  by a vector-valued transference technique based on Fendler's dilation and the vector-valued $q$-variational inequality for differential operators, which is another main result of this paper.

Let $f\in L^{p}(\mathbb{R}^d;\B)$, define
$$A_tf(x)=\frac{1}{|B_t|}\int_{B_t}f(x+y)dy,\;\forall t>0.$$
Then for any $1<p<\infty$ and Banach space $\B$ of martingale cotype $q_0$ with $2\leq q_0<\infty$, we have
\begin{align}\label{Z1 q variation p estimates0}
\|(A_tf)_{n\geq0}\|_{L^p(V_q(\B))}\leq C_{p,q}\|\|f\|\|_{p},\;\forall f\in L^p(\mathbb{R}^d;\B).
\end{align}
The weak type $(1,1)$ and $(L^{\infty}_c,BMO)$ estimates will also be shown in the process of obtaining (\ref{Z1 q variation p estimates0}). Here $L^{\infty}_c$ means compact supported $L^\infty$ functions. We prefer to state our results for $L^{\infty}_c$ but not $L^{\infty}$, because $q$-variation may behave badly at this end point. See \cite{Hon} for more information.
These results generalize the main results in \cite{JKRW98} \cite{JRW03} not only to the vector-valued case, but also to all $p>2$. Our strategy of the proof follows \cite{JKRW98} \cite{JRW03} but with additional analysis on geometric property of $\mathbb{R}^d$ and  more efforts on vector-valued analysis.

On the other hand, the vector-valued variational inequality is of independent interest, since the variational inequality (\ref{Z1 q variation p estimates0}) shares similar feature as Littlewood-Paley type inequality (\ref{littlewood-palye0}) in characterizing geometric property of Banach space. Precisely, we show that if inequality (\ref{Z1 q variation p estimates0}) holds for one $1<p<\infty$ (or equivalently all $1<p<\infty$), then Banach space $\B$ must be of Rademacher cotype $q$. However, this result is not sharp since in the martingale case, inequality (\ref{pisierxu}) implies obviously $\B$ must be of martingale cotype $q$ which is a stronger property than Rademacher cotype $q$. This inspire us a lot to consider vector-valued variational inequality in harmonic analysis and try to find new characterization of Banach space geometric property in subsequent works.

Our paper is organized as follows. In Section 2, we present the proof of inequality (\ref{Z1 q variation p estimates0}) and the related results. In Section 3, we prove the necessity of Rademacher cotype $q$ for inequality (\ref{Z1 q variation p estimates0}) to hold. In Section 4, by transference principle, we show variational inequality (\ref{ergodic average0}) for ergodic averages. In Section 5, the variational inequality (\ref{analytic semigroup continuous0}) for the semigroups itself will be proved, based on (\ref{ergodic average0}) and Littlewood-Paley inequality (\ref{littlewood-palye0}). In the last section, we establish individual (pointwise) ergodic theorems and the quantitative formulation of convergence.

\section{Differential operators}
Let $\B$ be a Banach space and $f:\;\R^d\rightarrow \B$. For $t>0$, let $B_t$ denote the the open ball centered at the origin $0$ with radius $r(B_t)$ equal to $t$. Then we define
  $$A_t f(x)= \frac{1}{|B_t|} \int_{B_t}f(x+y)~dy= \frac{1}{|B_t|} \int_{\R^d}f(y) \un_{B_t}(x-y)~dy, \quad x\in\R^d.$$
These are the central differential operators on $\R^d$. The $\B$-valued $q$-variation of the family of differential operators will be defined as
  $$\V_q(\A)f(x)=\|(A_tf(x))_{t\geq0}\|_{V_q(\B)}.$$

Jones, Rosenblatt and  Wierdl proved in \cite{JRW03} that the operator $\V_q(\A)$~$(q>2)$ is bounded on $L^p(\R^d)$ for $1<p\le 2$ and from $L^1(\R^d)$ into $L^{1,\infty}(\R^d)$. The following theorem extends their result not only to all $p>2$ but also to the vector valued case.

\begin{theorem}\label{Zd q variation}
Let $2\leq q_0<\infty$  and $\B$ is of martingale cotype $q_0$. Then the following statements are true for any $q_0<q\leq\infty$: ~{\rm{(i)}}. For any $1<p<\infty$, there exist a constant $C_{p,q}$ such that
   \begin{align}\label{Zd q variation p estimates}
     \|\V_q(\A)f\|_{p}\leq C_{p,q}\|\|f\|\|_{p},\;\forall f\in L^p(\R^d;\B);
   \end{align}
  ~{\rm{(ii)}}. If the estimate (\ref{Zd q variation p estimates}) is true for some $1<p_0<\infty$, then there exists a constant $C_{p_0,q}$ such that
   \begin{align}\label{Zd q variation 1 estimate}
    \big|\{\V_q(\A)f>\lambda\}\big|\leq \frac{C_{p_0,q}}{\lambda}\|\|f\|\|_{1},\;\forall f\in L^1(\R^d;\B)
   \end{align}
  for any $\lambda>0$, and
   \begin{align}\label{Zd q variation bmo estimate}
    \|\V_q(\A)f\|_{BMO_d}\leq {C_{p_0,q}}\|\|f\|\|_{\infty},\; \forall f\in L^\infty_c(\R^d;\B),
   \end{align}
  where $BMO_d$ is the dyadic BMO space.
\end{theorem}

\begin{remark}
 In the above theorem, the family $\{B_t\}_{t>0}$ of balls can be replaced by the family $\{Q_t\}_{t>0}$ of cubes, where $Q_t$ is the cube centered at the origin and having side length equal to $t$.
\end{remark}

The starting point of the proof is the estimate (\ref{Zd q variation p estimates}) in the case $p=q_0$. Then we prove the weak type $(1,1)$ estimate (\ref{Zd q variation 1 estimate}) and the $(L^{\infty}_c,BMO_d)$ estimate (\ref{Zd q variation bmo estimate}). Finally, the general type $(p,p)$ estimates for $1<p\neq q_0<\infty$ follow from interpolation.

As mentioned previously, our results are motivated by the similar results for vector valued martingales. Actually, the results for martingales are used in our proof. Let us recall the related martingale in this case. For $n\in\mathbb{Z}$, let $\sigma_n$ be the $n$-th dyadic $\sigma$ algebra. That is, $\sigma_n$ is generated by the dyadic cubes with side-length equal to $2^n$. Denote by $\E_n$ the conditional expectation with respect to $\sigma_n$.

Let $(f_n)_{n\in\Z}$ be a sequence of local integrable functions, measurably relative to the increasing sequence of $\sigma$-algebras $(\sigma_n)_{n\in\Z}$ and $df_n= f_{n-1}-f_n$ for $n\in \Z$.  $(f_n)_{n\in\Z}$ is said to a martingale if for each $n\in \Z$, we have $\E_{n}df_n=0$, and $df_n$ is called martingale difference.

\medskip
 Before proceeding with the proof of Theorem \ref{Zd q variation}, we need more notation. We will handle $\V_q(\A)$ by passing through long and short variations. For each interval $I_i=(t_i,t_{i+1}]$,  first consider two cases:
\begin{enumerate}[$\bullet$]
 \item Case 1: $I_i$ does not contain  any power of $2$;
 \item Case 2: $I_i$ contains powers of $2$.
\end{enumerate}
In case~1, $I_i\subset (2^k,\, 2^{k+1}]$ for some $k\in\Z$. In case~2, letting $m_i=\min\{k: 2^k\in I_i\}$ and $n_i=\max\{k: 2^k\in I_i\}$, we divide $I_i$ into three subintervals: $(t_i,\, 2^{m_i}],\, (2^{m_i},\, 2^{n_i}]$ and $(2^{n_i},\, t_{i+1}]$. Then we introduce two  collections of intervals:
\begin{enumerate}[$\bullet$]
 \item $\mathcal{S}$ consists of all intervals in case~1, and $(t_i,\, 2^{m_i}]$, $(2^{n_i},\, t_{i+1}]$ in case~2 (and $\mathcal{S}_k$ consists of all intervals in $\mathcal{S}$ and contained in $(2^k,\, 2^{k+1}]$);
 \item $ \mathcal{L}$ consists of all intervals  $(2^{m_j},\, 2^{n_j}]$ in case~2.
\end{enumerate}
Note that $\mathcal{S}\cup\mathcal{L}$ is a disjoint family of intervals. Then for any increasing sequence $(t_i)$, we have
\begin{align*}
  &\big(\sum_{i}\|A_{t_{i+1}} f(x)-A_{t_{i}} f(x)\|^{q}\big)^{\frac{1}{q}}
    \leq C_q\big(\sum_{I_i\in\mathcal{S}}\|A_{t_{i+1}} f(x)-A_{t_{i}} f(x)\|^{q} \big)^{\frac{1}{q}}\\
+ &C_q \big(\sum_{I_i\in\mathcal{L}}\|A_{t_{i+1}} f(x)-A_{t_{i}} f(x)\|^{q} \big)^{\frac{1}{q}}=C_q(I+II).
\end{align*}
The first term  on the right hand side is controlled by
  $$I \leq \sup_{(t_i)_i}\big( \sum_{k\in \Z}\sum_{I_i\in \mathcal{S}_k}
    \|A_{t_{i+1}}f(x)- A_{t_{i}}f(x)\|^{q_0}\big)^{\frac{1}{q_0}}$$
which is denoted by $\mathcal{SV}_{q_0}(\A)f$. While the second term is further estimated  in the following. Note that by the definition of $\mathcal L$, if $I_i=(t_i,t_{i+1}]\in \mathcal L$, then we can rewrite $I_i$ as $(2^{n_i},2^{n_{i+1}}]$.
\begin{align*}
II& =\big(\sum_{I_i\in\mathcal{L}}\|A_{t_{i+1}} f(x)-A_{t_{i}} f(x)\|^{q} \big)^{\frac{1}{q}} \\
  &\leq\big(\sum_{I_i\in\mathcal{L}}\| A_{t_{i+1}} f(x)-\E_{{n_{i+1}}}f(x)+\E_{{n_{i}}}f(x)
      - A_{t_{i}}f(x)\|^{q}\big)^{\frac{1}{q}}\\
  &\quad +\big(\sum_{I_i\in\mathcal{L}}\|\E_{{n_{i+1}}}f(x)-\E_{{n_{i}}}f(x)\|^{q} \big)^{\frac{1}{q}}\\
  &\leq\big(\sum_{I_i\in\mathcal{L}}\| A_{2^{n_{i+1}}}f(x)-\E_{{n_{i+1}}}f(x)+\E_{{n_{i}}}f(x)
      - A_{2^{n_{i}}}f(x)\|^{q_0} \big)^{\frac{1}{q_0}}\\
  &\quad +\sup_{(n_i)}\big(\sum_{i}\|\E_{{n_{i+1}}}f(x)-\E_{{n_{i}}}f(x)\|^{q}\big)^{\frac{1}{q}}=III+IV,
\end{align*}
Note that the term $IV$ is just the vector-valued $q$-variation for martingales in \cite{PiXu88}, denoted by $\V_{q}(\E)f$. On the other hand, by the triangle inequalities,  the term $III$ is controlled by
  $$\big(\sum_{n\in \Z} \|A_{2^n}f(x)-\E_{n}f(x)\|^{q_0}\big)^{\frac{1}{q_0}},$$
which is denoted by $\mathcal{LV}_{q_0}(\A)f$.

To conclude, we have
\begin{align}\label{vqm by svq0m lvq0m vqe}
\V_q(\A)f\leq \mathcal{SV}_{q_0}(\A)f+\mathcal{LV}_{q_0}(\A)f+\V_q(\E)f.
\end{align}
This inequality enable us to deduce (\ref{Zd q variation p estimates}) and (\ref{Zd q variation 1 estimate}) from the corresponding estimates for $\mathcal{SV}_{q_0}(\A)$ and $\mathcal{LV}_{q_0}(\A)$.
The following theorem deals with the long variation operator $\mathcal{LV}_{q_0}(\A)$.
\begin{theorem}\label{Zd long variation}
Let $2\leq q_0<\infty$ and $\B$ is of martingale cotype $q_0$. Then: {\rm (i)}. For any $1<p<\infty$, there exist a constant $C_{p,q_0}$ such that
\begin{align}\label{Zd long variation p estimates}
\|\mathcal{LV}_{q_0}(\A)f\|_{p}\leq C_{p,q_0}\|\|f\| \|_{p},\;\forall f\in L^p(\R^d;\B);
\end{align}
{\rm{(ii)}}. If the estimate (\ref{Zd long variation p estimates}) is true for some $1<p_0<\infty$, then there exists a constant $C_{p_0,q_0}$ such that
\begin{align}\label{Zd long variation 1 estimate}
\big|\{\mathcal{LV}_{q_0}(\A)f>\lambda\}\big|\leq \frac{C_{p_0,q_0}}{\lambda}\|\|f\|\|_{1},\;\forall f\in L^1(\R^d;\B).
\end{align}
for any $\lambda>0$, and
\begin{align}\label{Zd long variation bmo estimate}
\|\mathcal{LV}_{q_0}(\A)f\|_{BMO_d}\leq {C_{p_0,q_0}}\|\|f\|\|_{\infty},\;\forall f\in L^{\infty}(\Z^d;\B).
\end{align}
\end{theorem}
The same conclusions hold true for the short variation operator $\mathcal{SV}_{q_0}(\A)$.
\begin{theorem}\label{Zd short variation}
Let $2\leq q_0<\infty$ and $\B$ is of martingale cotype $q_0$. Then: {\rm (i)}. For any $1<p<\infty$, there exist a constant $C_{p,q_0}$ such that
\begin{align}\label{Zd short variation p estimates}
\|\mathcal{SV}_{q_0}(\A)f\|_{p}\leq C_{p,q_0}\|\|f\|\|_{p},\;\forall f\in L^p(\R^d;\B);
\end{align}
{\rm{(ii)}}. If the estimate (\ref{Zd short variation p estimates}) is true for some $1<p_0<\infty$, then there exists a constant $C_{p_0,q_0}$ such that
\begin{align}\label{Zd short variation 1 estimate}
\big|\{\mathcal{SV}_{q_0}(\A)f>\lambda\}\big|\leq \frac{C_{p_0,q_0}}{\lambda}\|\|f\|\|_{1},\;\forall f\in L^1(\R^d;\B).
\end{align}
for any $\lambda>0$, and
\begin{align}\label{Zd short variation bmo estimate}
\|\mathcal{SV}_{q_0}(\A)f\|_{BMO}\leq {C_{p_0,q_0}}\|\|f\|\|_{\infty},\;\forall f\in L^{\infty}(\R^d;\B).
\end{align}
\end{theorem}

As the comments after Theorem \ref{Zd q variation}, we shall prove firstly in Theorem \ref{Zd long variation} and Theorem \ref{Zd short variation} the $(q_0,q_0)$ estimate and weak type $(1,1,)$ estimate, which rely on the following almost orthogonality principle. 
\begin{lemma}\label{almost orthogonality}
Suppose $(\Su_n)_{n\in\Z}$ is a sequence of subadditive operators from $L^{q_0}(\B)$ to $L^{q_0}$ in some $\sigma$-finite measure space, i.e. $\Su(f+g)\leq\Su(f)+\Su(g)$. Let $(u_n)_{n\in\mathbb{Z}}$ and $(v_n)_{n\in\Z}$ be two sequences of $L^{q_0}(\B)$ functions. Then we have
\begin{align*}
\sum_k\|\sup_{j,m}\Su_k(\sum_{j\leq n\leq m}u_n)\|^{q_0}_{q_0}\leq w^{q_0}\cdot\sum_{n}\|\|v_n\|\|^{q_0}_{q_0}
\end{align*}
provided that there is a sequence $(\sigma(j))_{j\in\Z}$ of positive numbers with $w=\sum_{j}\sigma(j)<\infty$ such that
$$\|\Su_ku_n\|_{q_0}\leq\sigma(n-k)\|\|v_n\|\|_{q_0}$$
for every $n,k$.

Furthermore, if~$\Su$ are continuous, $f=\sum_nu_n$, $\sum_n\|\|v_n\| \|^{q_0}_{q_0}\leq C\|\|f\|\|^{q_0}_{q_0}$, then
  $$\sum_k\|\Su_kf\|^{q_0}_{q_0}\leq Cw^{q_0}\cdot\|\|f\|\|^{q_0}_{q_0}.$$
\end{lemma}
\begin{proof}
By the subadditivity of the $\Su$,
\begin{align*}
\sup_{j,m}\Su_k(\sum_{j\leq n\leq m}u_n)&\leq\sup_{j,m}\sum_{j\leq n\leq m}\Su_k(u_n)=\sum_n\Su_ku_n.
\end{align*}
Hence by the triangle inequality for the $L^{q_0}$ norm,
  $$\|\sup_{j,m}\Su_k(\sum_{j\leq n\leq m}u_n)\|_{q_0}\leq\sum_n\|\Su_ku_n\|_{q_0}.$$
Setting $b_n=\|\|v_n\| \|_{q_0}$, by the assumption $\|\Su_ku_n\|_{q_0}\leq\sigma(n-k)\|\|v_n\|\|_{q_0}$,
it suffices to prove
$$\sum_k\big(\sum_n\sigma(n-k)b_n\big)^{q_0}\leq\big(\sum_n\sigma(n)\big)^{q_0}\cdot\sum_{n}b^{q_0}_n,$$
which follows from the well known inequality for the $\ell^{q_0}$ norm of the convolution of two sequences
$$\|\sigma\ast b\|_{\ell^{q_0}}\leq\|\sigma\|_{\ell^1}\|b\|_{\ell^{q_0}}.$$
\end{proof}

\subsection{Strong type $(q_0,q_0)$ estimates}
By (\ref{vqm by svq0m lvq0m vqe}), the assertion that $\V_{q}(\A)$ is bounded from $L^{q_0}(\B)$ to $L^{q_0}$ is  an immediate consequence of the strong type $(q_0,q_0)$ estimates of $\mathcal{LV}_{q_0}(\A)$,  $\mathcal{SV}_{q_0}(\A)$ and $\V_{q}(\E)$. Now let us first prove strong $(q_0,q_0)$ estimate of $\mathcal{LV}_{q_0}(\A)$.
\begin{proof}
Let us write $f=\sum_n d_n$, where $d_n=\E_{n-1}f-\E_nf$ for $n\in\Z$. Then $\B$ being of martingale cotype $q_0$ means
  $$\sum_n\|\|d_n\|\|^{q_0}_{q_0}\leq C_{q_0}\|\|f\|\|^{q_0}_{q_0}.$$

Taking $\Su_kg(x)=\|A_{2^k}g(x)-\E_k g(x)\|$, $u_n=v_n=d_n$ and $\sigma(j)=C2^{-|j|/{q_0}}$ in Lemma \ref{almost orthogonality},  it suffices to prove
\begin{align}\label{1 esti in proof of LV}
\|\|(A_{2^k}-\E_k)d_n\| \|^{q_0}_{q_0}\leq C\cdot 2^{-|n-k|}\|\|d_n\| \|^{q_0}_{q_0}.
\end{align}
We first prove (\ref{1 esti in proof of LV}) in the case $k<n$. Note that $\E_kd_n=d_n$ when $k<n$. It suffices to show
\begin{align}\label{2 esti in proof of LV}
\|\|A_{2^k}d_n-d_n\|\|^{q_0}_{q_0}\leq C\cdot2^{(k-n)}\|\|d_n\|\|^{q_0}_{q_0}.
\end{align}
Let us denote by $\mathbb{D}_n$ the set of all atoms of the $n$-th dyadic $\sigma$ algebra $\sigma_n$, and write
  $$\|\| A_{2^k}d_n-d_n\|\|^{q_0}_{q_0}=\int_{\R^d}\|A_{2^k}d_n-d_n\|^{q_0}
    =\sum_{H\in\mathbb{D}_{n-1}}\int_H\|A_{2^k}d_n-d_n\|^{q_0}.$$
Since $d_n$ is constant on the atom $H\in\mathbb{D}_{n-1}$, we have $A_{2^k}d_n-d_n=0$ if $x+ B_{2^k}\in H$. Hence for $x\in H$, $(A_{2^k}d_n-d_n)(x)$ may be nonzero only if $x+B_{2^k}$ intersects with the complement of $H$. Hence for a given set $B\subset \R^d$, we consider the set
  $$\mathcal{H}(B,H)= \{x\in H|~x+ B\cap H^C\neq \emptyset\}.$$
Since $H\in\mathbb{D}_{n-1}$, we have
  $$|\mathcal{H} (B_{2^k},H)|\leq C2^{(d-1)n}\cdot 2^k.$$
Denoting by $m_H$ the maximum of $\|d_n\|$ on the cubes neighboring $H$ and $H$, then we have the estimate $\|(A_{2^k}d_n(x)-d_n(x)\| \leq 2m_H$ for every $x\in H$. Hence
\begin{align*}
 \int_H \|A_{2^k}d_n-d_n\|^{q_0}&\leq C2^{(d-1)n}\cdot2^k\cdot m_H^{q_0}\leq C2^{k-n} \int_Hm_H^{q_0}.
\end{align*}
Summing over all $H\in\mathbb{D}_{n-1}$ and noting that $\int_{\R^d}m^{q_0}_H \leq  C_d\int_{\R^d}\|d_n\|^{q_0}$ since $m_H$ is a constant on $H$. We finish the proof of (\ref{2 esti in proof of LV}).

Now let us prove (\ref{1 esti in proof of LV}) in the case $n\leq k$. Note that $\E_kd_n=0$ in this case, hence it suffices to prove
\begin{align}\label{3 esti in proof of LV}
\|\|A_{2^k}d_n\|\|^{q_0}_{q_0}\leq C2^{n-k}\|\|d_n\|\|^{q_0}_{q_0},
\end{align}
which is deduced  from the following pointwise estimate by integrating both sides:
\begin{align}\label{4 esti in proof of LV}
 \|A_{2^k} d_n \|^{q_0}\leq C 2^{n-k}\cdot A_{2^k}\|d_n\|^{q_0}.
\end{align}
We finish the proof once we prove the inequality (\ref{4 esti in proof of LV}). Indeed, integrating over $\R^d$, we get
\begin{align*}
 \int_{\R^d} \|A_{2^k}d_n \|^{q_0} &\leq C2^{n-k} \int_{\R^d} \frac1{|B_{2^k}|}\int_{x+B_{2^k}}\|d_n(y)\|^{q_0}dydx\\
     &=C2^{n-k} \frac1{|B_k|}\int_{B_{k}}dy\int_{\R^d}\|d_n(x+y)\|^{q_0}dx\\
     &=C2^{n-k}\|\|d_n\|\|^{q_0}_{q_0}.
\end{align*}

We divide $\R^d$ into all atoms in $\mathbb{D}_n$. Taking any $H\in\mathbb{D}_n$, we have the fact $\int_Hd_n=0$. Then for a given set $B\subset \R^d$, we consider
  $$ \mathcal{I}(x+B,n)=\cup\{(x+B)\cap H| H\in \mathbb{D}_n,~x+\partial B\cap H\neq \emptyset\}.$$
Since the $\{H|H\in \mathbb{D}_n\}$ are pairwise disjoint, we have
\begin{align*}
  \int_{x+B_{2^k}} d_n =\sum_{H\in\mathbb{D}_n} \int_{B_{2^k}+x\cap H}d_n\leq \int_{\mathcal{I}(x+B_{2^k},n)} d_n.
\end{align*}
Now since the measure of $\mathcal{I}(x+B_{2^k},n)$ is not more than a constant multiple of $2^n 2^{(d-1)k}$ and $|B_{2^k}|\approx 2^{kd}$,  using H\"older inequality, we get
\begin{align*}
  \|A_{2^k}d_n \|^{q_0} &=\frac1{|B_{2^k}|^{q_0}} \big\|\int_{\mathcal{I}(x+B_{2^k},n)}d_n\big\|^{q_0}\\
     &\leq C\frac{1}{(2^{dk})^{q_0}} (2^{n+(d-1)k})^{q_0-1} \int_{x+B_{2^k}}\|d_n\|^{q_0}\\
    &\leq C( 2^{n-k})^{q_0-1} \big(\frac1{2^{dk}}  \int_{x+B_{2^k}}\|d_n\|^{q_0}\big)\\
     &\leq C 2^{n-k}  \big(\frac1{|B_{2^k}|} \int_{x+B_{2^k}}\|d_n\|^{q_0}\big).
\end{align*}
Then we prove inequality (\ref{4 esti in proof of LV}).
\end{proof}

The structure of the proof of strong $(q_0,q_0)$ estimate for  $\mathcal{SV}_{q_0}(\A)$  in Theorem \ref{Zd short variation} is similar to $\mathcal{LV}_{q_0}(\A)$.
\begin{proof}
Let $\Su_k$ be such as
$$\Su_kg(x)=\big(\sum_{I_i\in\mathcal{S}_k}\|(A_{t_{i+1}}-A_{t_{i}})g\|^{q_0}\big)^{\frac{1}{q_0}}(x),$$
and $u_n=v_n=d_n$ and $\sigma(j)=C2^{-|j|/{q_0}}$ in Lemma \ref{almost orthogonality}. Note that $I_i \in\mathcal{S}_k$ means $I_i= (t_i,t_{i+1}]\subset (2^{k},2^{k+1}]$.
Then it suffices to prove
\begin{align}\label{1 esti in proof of SV}
\int_{\R^d}\sum_{I_i\in\mathcal{S}_k}\|(A_{t_{i+1}}-A_{t_i})d_n\|^{q_0}\leq C2^{-|n-k|}\|\|d_n\|\|^{q_0}_{q_0}.
\end{align}
We first prove (\ref{1 esti in proof of SV}) in the case $k<n$. We then have
$$\int_{\R^d}\sum_{I_i\in\mathcal{S}_k}\|(A_{i+1}-A_{t_i})d_n\|^{q_0} =\sum_{H\in\mathbb{D}_{n-1}}
\int_{H}\sum_{I_i\in\mathcal{S}_k}\|(A_{B_{t_{i+1}}}-A_{B_{t_{i}}})d_n\|^{q_0}.$$
Since $d_n$ is constant on $H\in \mathbb{D}_{n-1}$,
$\sum_{I_i\in\mathcal{S}_k}\|(A_{t_{i+1}}-A_{t_{i}})d_n\|^{q_0}$
can be nonzero only if for some $I_i\in\mathcal{S}_k$ at least the ball $x+B_{t_{i+1}}$ intersects the complement of $H$. Noticing that $B_{t_{i+1}}\subset B_{2^{k+1}}$, we consider the set
  $$\mathcal{H}(B_{2^{k+1}},H)= \{x\in H|~x+B_{2^{k+1}}\cap H^C\neq \emptyset\}.$$
And the measure of $\mathcal{H}(B_{2^{k+1}},H)$ is not more than a constant multiple of $2^{(d-1)n}2^{k}$.
Recall that the maximum of $\|d_n\|$ on the cubes neighboring $H$ and $H$ is denoted by $m_H$.
Hence for every $x\in H$, we have the estimate as follows:
\begin{align*}
    &\big(\sum_{I_i\in\mathcal{S}_k}\|(A_{t_{i+1}}-A_{t_i})d_n\|^{q_0}\big)^{\frac{1}{q_0}}(x)
        \leq \sum_{I_i\in\mathcal{S}_k}\|(A_{t_{i+1}}-A_{t_i})d_n\|(x)\\
\leq& C\frac{1}{|B_{2^{k}}|} \sum_{I_i\in\mathcal{S}_k} \int_{x+(B_{2^{k+1}}\setminus B_{2^{k}}) }\|d_n\|
        +\sum_{I_i\in\mathcal{S}_k}\big(\frac{1}{|B_{t_{i}}|}-\frac{1}{|B_{t_{i+1}}|}\big)\int_{x+B_{2^{k+1}}}\|d_n\|\\
\leq& C\frac{1}{|B_{2^{k}}|}\int_{x+B_{2^{k+1}}}\|d_n\| +\frac{1}{|B_{2^{k}}|}\int_{x+B_{2^{k+1}}}\|d_n\|\\
\leq& C A_{2^{k+1}} \|d_n\|\leq Cm_H.
\end{align*}
Then we have
\begin{align*}
\int_{H}\sum_{I_i\in\mathcal{S}_k}\|(A_{t_{i+1}}-A_{t_{i}})d_n\|^{q_0}
   &\leq C \int_{\mathcal{H}(B_{2^{k+1}},H)}  m_H^{q_0} \leq C2^{(d-1)n}2^{k}  m_H^{q_0}\\
   &=C2^{k-n}( 2^{dn} m_H^{q_0}) \leq C2^{k-n}\int_{H}m_H^{q_0}.
\end{align*}
Summing over $H\in\mathbb{D}_{n-1}$ and noting that $\int_{\R^d}m^{q_0}_n\leq C\int_{\R^d}\|d_n\|^{q_0},$ we finish the prove of (\ref{1 esti in proof of SV}) in the case $n>k$.

Now we turn to the proof of (\ref{1 esti in proof of SV}) in the case $n\leq k$. It suffices to prove the pointwise estimate
\begin{align}\label{2 esti in proof of SV}
 \sum_{I_i\in\mathcal{S}_k}\|(A_{t_{i+1}}-A_{t_{i}})d_n(x)\|^{q_0}\leq C2^{n-k}A_{2^{k+1}}\|d_n(x)\|^{q_0},~x\in \R^d,
\end{align}
since then we can immediately get (\ref{1 esti in proof of SV}) by integrating both sides on $\R^d$ as we have done for the long $q_0$-variation. We first deduce
\begin{align*}
 \sum_{I_i\in\mathcal{S}_k}&\|(A_{t_{i+1}}-A_{t_{i}})d_n\|^{q_0} \leq C_{q_0}\sum_{I_i\in\mathcal{S}_k}
       \Big\|\frac{1}{|B_{t_{i+1}}|}\int_{x+B_{t_{i+1}}\setminus B_{t_{i}}}d_n\Big\|^{q_0}\\
     &+C_{q_0}\sum_{I_i\in\mathcal{S}_k}\Big(\frac{1}{|B_{t_{i}}|}-\frac{1}{|B_{t_{i+1}}|}\Big)^{q_0}
       \Big\|\int_{x+B_{t_{i}}}d_n\Big\|^{q_0}\\
\leq & C_{q_0}\frac{1}{|B_{2^k}|^{q_0}}\Big\|\sum_{I_i\in\mathcal{S}_k}\int_{x+B_{t_{i+1}}\setminus B_{t_{i}}}d_n\Big\|^{q_0}\\
     & +C_{q_0}\sum_{I_i\in\mathcal{S}_k}\Big(\frac{1}{|B_{t_{i}}|}-\frac{1}{|B_{t_{i+1}}|}\Big)^{q_0}
    \sup_{I_i\in \mathcal{S}_k}\Big\|\int_{x+B_{t_{i}}}d_n\Big\|^{q_0}\\
   = & I+II.
\end{align*}
In the first term $I$, we need to consider the set $\mathcal{I}(x+B_{t_{i+1}}\setminus B_{t_{i}},n)$ for any $I_i\in \mathcal{S}$, which is defined as before. By the fact that $\int_H d_n=0$ for any $H\in \mathbb{D}_n$, we first have
$$\int_{x+B_{t_{i+1}}\setminus B_{t_i}} d_n=\int_{\mathcal{I}(x+B_{t_{i+1}}\setminus B_{t_{i}},n)}d_n.$$
Since $B_{t_{i+1}}\setminus B_{t_i} \subset B_{2^{k+1}}\setminus B_{2^k}$, the measure of $\mathcal{I}(x+B_{t_{i+1}}\setminus B_{t_i},n)$ is not more than a constant multiple of $2^n 2^{(d-1)k}$.
Noting $|B_{2^k}|\approx 2^{kd}$ and using the H\"older inequality, we estimate the first term as
\begin{align*}
 I & \leq C\frac1{(2^{k d})^{q_0}} (2^n 2^{(d-1)k})^{q_0-1} \sum_{I_i\in\mathcal{S}_k}
     \int_{x+B_{t_{i+1}}\setminus B_{t_i}}\|d_n\|^{q_0}\\
   & \leq C (2^{n-k})^{q_0-1}\big(\frac1{2^{k d}} \int_{x+B_{2^{k+1}}}\|d_n\|^{q_0}\big)\\
   & \leq C 2^{n-k}  \big(\frac1{|B_{2^{k+1}}|}\int_{x+B_{2^{k+1}}} |d_n|_{\B}^{q_0}\big).
\end{align*}
In the second term $II$, we have
$$\sum_{I_i\in\mathcal{S}_k}\Big(\frac{1}{|B_{t_{i}}|}-\frac{1}{|B_{t_{i+1}}|}\Big)^{q_0}
  \leq  \Big(\frac{1}{|B_{2^k}|}-\frac{1}{|B_{2^{k+1}}|}\Big)^{q_0}\leq \frac{C}{(2^{kd})^{q_0}}.$$
Then noticing that the measure of $\mathcal{I}(x+B_{t_{i}},n)$ is also not more than a constant multiple of $2^n 2^{(d-1)k}$, and using H\"older inequality and running randomly $A_{t_i}$ for $I_i\in\mathcal{S}_k$, the second term then is estimated as
\begin{align*}
 II  & \leq C\frac1{(2^{dk})^{q_0}} (2^n 2^{(d-1)k})^{q_0-1}\sup_{I_i\in \mathcal{S}_k} \int_{x+B_{t_i}}\|d_n\|^{q_0}\\
     &\leq C 2^{n-k} \big(\frac1{|B_{k+1}|}\int_{x+B_{k+1}} \|d_n\|^{q_0}\big).
\end{align*}
Combining the estimates of $I$ and $II$, we have proved the operator $\mathcal{S}\V_{q_0}(\A)$ is of type $(q_0,q_0)$.
\end{proof}

\subsection{Weak type $(1,1)$ estimates}
By (\ref{vqm by svq0m lvq0m vqe}), the assertion that $\V_{q}(\A)$ is bounded from $L^1(\B)$ to $L^{1,\infty}$ is  an immediate consequence of the weak type $(1,1)$ estimates of $\mathcal{LV}_{q_0}(\A)$,  $\mathcal{SV}_{q_0}(\A)$ and $\V_{q}(\E)$.

As usually, the proofs of the weak type $(1,1)$ inequalities in Theorem \ref{Zd long variation} and Theorem \ref{Zd short variation} follow from the scheme of Calder\'on-Zygmund (cf. e.g. Theorem~II.1.12 in \cite{gar-rubio}). However, as in \cite{JRW03}, we will estimate the $L^{q_0}$ norm of the bad function off the set where the maximal function is large. The following vector-valued Calder\'on-Zygmund decomposition, which plays a key role in proving weak type estimate, should be known somewhere but I can not find it in any literature.
Given  a cube $Q\subset\R^d$, let $Q^*$ denote the cube with the same center as $Q$ but three times the side length.

\begin{lemma}\label{C Z decomposition}
 Let   $f\in L^1(\R^d;\B)$ and  $\lambda >0$.  Then there exists a finite  disjoint family $\{Q_i\}$ of dyadic cubes satisfying the following properties
\begin{itemize}
 \item[\rm(i)] $\|f\| \le \lambda$ on $\Omega^c$, where $\displaystyle\Omega = \bigcup_i Q_i$;
 \item[\rm(ii)] $\displaystyle \lambda < \frac1{|Q_i|} \int_{Q_i} \|f\| \le 2^d\lambda$;
 \item[\rm(iii)] $\displaystyle\Omega\subset\{x\in\R^d: M(\|f\|)(x) > \lambda\}$
                 and $\displaystyle\{x\in\R^d: M(\|f\|)(x) >4^d\lambda\}\subset  \Omega^*$,
                 where $\displaystyle \Omega^* = \bigcup_i Q_i^*$;
 \end{itemize}
Define
\begin{align*}
 &g=f \textrm{ on } \Omega^c\quad\textrm{and}\quad g= \frac1{|Q_i|} \int_{Q_i} f \, \textrm{ on } Q_i  \textrm{ for each }i,\\
 &b= \sum_i b_i, \textrm{ where } b_i = \big(f -\frac1{|Q_i|} \int_{Q_i} f  \big) \un_{Q_i} .
\end{align*}
Then
\begin{itemize}
 \item[\rm(iv)] $f=g+b$;
 \item[\rm(v)] $\|\|g\| \|_\infty \le 2^d\lambda$ and $\|\|g\| \|^{p}_{p}\leq 2^{d(p-1)} \lambda^{p-1}\|\|f\|\|_1$, for any $1\leq p<\infty$;
 \item[\rm(vi)] for each $i$, $\displaystyle  \int_{\R} b_i= 0 $
                and $\displaystyle \frac1{|Q_i|}\int_{\R}\|b_i\| \le 2^{d+1} \lambda$.
 \end{itemize}
\end{lemma}
Note that the cubes $Q_j$'s in the lemma are selected by applying standard stopping arguments
to the scalar-valued function $\|f\|$. Thus for any $Q_j$, there exists $n\in \Z$, such that $Q_j\in \mathbb{D}_n$. In the rest of this paper, we denote by $\Q$ the collection of all $Q_j$'s, and the collections
$\Q_n=\mathbb{D}_n \cap \Q$, for all $n\in \Z$.

Now let us give the proof of the weak type $(1,1)$ estimate in Theorem \ref{Zd long variation}.
\begin{proof}
Let $\lambda>0$. By triangle inequality, keeping the notation in Lemma  \ref{C Z decomposition}, we have
 $$|\{\mathcal{LV}_{q_0}(\A)f>\lambda\}|\leq|\{\mathcal{LV}_{q_0}(\A)g>\lambda/2\}|+|\{\mathcal{LV}_{q_0}(\A)b>\lambda/2\}|.$$
The first term on the right is estimated by the $L^{q_0}$-boundedness of $\mathcal{LV}_{q_0}(\A)$  and (v)
in Lemma  \ref{C Z decomposition}:
\begin{align*}
 |\{\mathcal{LV}_{q_0}(\A)g>\lambda/2\}|& \leq  C_{q_0} \lambda^{-q_0}\int_{\R^d} |\mathcal{LV}_{q_0}(\A)g(x)|^{q_0}dx\\
    &\leq  C_{q_0} \lambda^{-q_0}\|\|g\|\|^{q_0}_{q_0}\leq C_{q_0} \lambda^{-1}\|\|f\|\|_1.
\end{align*}
Our main task is to prove a similar estimate for the bad part $b$. We have
 $$|\{\mathcal{LV}_{q_0}(M)b>\lambda/2\}|\leq|\Omega^*|+|(\R^d\setminus \Omega^*)\cap\{\mathcal{LV}_{q_0}(M)b>\lambda/2\}|.$$
The first term is estimated by Lemma \ref{C Z decomposition} (ii) as
 $$|\Omega^*_j|\leq \sum_j|Q^*_j|\leq C_d\sum_j|Q_j|\leq C_d
   \lambda^{-1} \sum_j\int_{Q_j} \|f(x)\|dx \leq C_d \lambda^{-1} \|\|f\|\|_1.$$
It remains to treat the second term. By Chebyshev's inequality, we first have
\begin{align*}
 |(\R^d\setminus\Omega^*)\cap\{\mathcal{LV}_{q_0}(\A)b>\lambda/2\}|
    &\leq C_{q_0} \lambda^{-q_0} \int_{\R^d\setminus\Omega^*}|\mathcal{LV}_{q_0}(\A)b(x)|^{q_0}dx\\
    & =C_{q_0} \lambda^{-q_0}  \int_{\R^d\setminus\Omega^*}\sum_k \|A_{2^k}b-\E_kb(x)\|^{q_0}dx.
\end{align*}

In order to finish the proof, we shall use again the almost orthogonality principle Lemma \ref{almost orthogonality}.
Write $b$ as
  $$b=\sum_n \sum_{j:Q_j\in\Q_n} b_j=\sum_n h_n.$$

Clearly, $x\notin Q^*$ implies $\E_k b=0$, for every $k$.
Indeed, if $k\leq n$, then $\E_k h_n(x)=0$ since the atom of $\sigma_k$ containing $x$ is disjoint from the support of $h_n$; if $n<k$, for each $Q_j\in \Q_n$, there exists an atom of $\sigma_k$ containing $Q_j$. Then by Lemma \ref{C Z decomposition}, (vi), we get  $\E_kb_j(x)=0$. Hence $\E_kh_n(x)=0$.

On the other hand, $A_{2^k}h_n(x)=0$ for $k\leq n$. This is because $x+B_{2^k}$ is disjoint from any of the $Q_j$ in the support of $h_n$.
By these discussions, it suffices to prove
  $$\sum_k \int_{\R^d\setminus \Omega^*}\|A_{2^k}\sum_{n<k}h_n\|\leq{C_{q_0}}{\lambda^{q_0}}\sum_j|Q_j|.$$
since we have $\sum_j|Q_j|\leq \|\|f\|\|_1$ by using  Lemma \ref{C Z decomposition} again. Set $d_n=\sum_{Q_j\in\Q_n}\un_{Q_j}$, Then we have
  $$\sum_n\|d_n\|^{q_0}_{q_0}=\sum_n\|d_n\|_1=\sum_n\sum_{Q_j\in\Q_n} |Q_j|.$$
Hence takeing  $\mathbb{S}_kg =|A_{2_k}g\|$, $u_n=h_n$, $v_n=d_n$ and $\sigma(j)=C2^{-|j|/{q_0}}$ in Lemma \ref{almost orthogonality}, it suffices to prove for $k>n$
  $$\|\|A_{2^k}h_n\|\|^{q_0}_{q_0}\leq C_{q_0}2^{n-k}\lambda^{q_0}\|d_n\|^{q_0}_{q_0}.$$
This follows by integrating both sides of the following pointwise inequality
\begin{align}\label{5 esti in proof of LV}
 \|A_{2^k}h_n\|^{q_0}\leq C_{q_0}2^{n-k}\lambda^{q_0}A_{2^{k+1}}|d_n|
\end{align}
using the fact $d_n=d_n^{q_0}$. Let us give a notation which we have used several times:
 $$\mathcal{I}(x+B_{2^k},n)=\cup\{ x+B_{2^k}\cap Q_i | Q_i\in \Q_n,~x+\partial B_{2^k}\cap Q_i\neq \emptyset\}.$$
Then by (ii) in Lemma \ref{C Z decomposition} and noting that $n<k$, we have
\begin{align*}
 \big\| \int_{x+B_{2^k}}h_n(x)dx \big\|& =\int_{\mathcal{I}(B_{2^k})}\|h_n(x)\| dx
       \leq \sum_{\substack{Q_i\in \Q_n\\  x+\partial B_{2^k}\cap Q_i\neq \emptyset}}\int_{Q_i}\|h_n(x)\| dx\\
 & \leq  2\cdot 2^d\lambda \sum_{\substack{Q_i\in \Q_n\\  x+\partial B_{2^k}\cap Q_i\neq \emptyset}} Q_i
       =  C \lambda \int_{x+B_{2^{k+1}}}|d_n|.
\end{align*}
Since we know that the measure of $\mathcal{I}(x+B_{2^k},n)$ is not more than a constant multiple of $2^n2^{(d-1)k}$, we then give another estimate
\begin{align*}
\big\|\frac{1}{|B_{2^k}|}\int_{x+B_k}h_n(x)dx \big\| & \leq \frac{1}{|B_{2^k}|}\int_{\mathcal{I}(B_{2^k})}\|h_n(x)\| dx
      \leq \frac{C \lambda}{|B_{2^k}|}\int_{\mathcal{I}(B_{2^k})}|d_n(x)| dx\\
 & \leq C \frac{\lambda |\mathcal{I}(B_{2^k})|}{|B_{2^k}|} \leq C \frac{\lambda 2^n 2^{(d-1)k}}{2^{kd}}
      \leq C 2^{n-k} \lambda.
\end{align*}
Putting the above two estimates together,  we then deduce (\ref{5 esti in proof of LV}) as follows:
\begin{align*}
 \|A_{2^k}h_n\|^{q_0}& =\|A_{2^k} h_n\|^{q_0-1}\|A_{2^k}h_n\|\\
   & \leq C 2^{(n-k)(q_0-1)}\lambda^{q_0}A_{2^{k+1}}|d_n|\leq C 2^{n-k}\lambda^{q_0}A_{2^{k+1}} |d_n|.
\end{align*}
Then we finish the proof of weak $(1,1)$ of the long variation operator.
\end{proof}

The proof of weak type (1,1) estimate (\ref{Zd short variation 1 estimate}) in Theorem \ref{Zd short variation} is a variant of that in Theorem \ref{Zd long variation}. Let us explain it briefly.
\begin{proof}
By the fact that $\mathcal{SV}_{q_0}(A)$ is of strong type $(q_0,q_0)$, it suffices to prove
 $$\int_{\R^d\setminus\Omega^*} \sum_k\sum_{I_i\in\mathcal{S}_k} \|A_{t_{i+1}}b-A_{t_i} b\|^{q_0}
    \leq{C_{q_0}}{\lambda^{q_0}}\sum_j|Q_j|.$$
By an argument similar to the one given in the proof of Theorem \ref{Zd long variation},  we only need to show the pointwise estimate
\begin{align}\label{3 esti in proof of SV}
\sum_{I_i\in\mathcal{S}_k}\|A_{t_{i+1}}h_n-A_{t_i}h_n\|^{q_0}\leq C_{q_0}2^{n-k}\lambda^{q_0}A_{2^{k+2}}|d_n|.
\end{align}
The left hand side can be controlled by, modulo a constant depending on $q_0$, the sum of
\begin{align*}
 J_1=\sum_{I_i\in\mathcal{S}_k}\big\|\frac{1}{|B_{t_{i+1}}|}\int_{x+(B_{t_{i+1}}\setminus B_{t_i}) }h_n\big\|^{q_0}
\end{align*}
and
\begin{align*}
J_2=\sum_{I_i\in\mathcal{S}_k}\big(\frac{1}{|B_{t_i}|}-\frac{1}{|B_{t_{i+1}}|}\big)^{q_0}
  \big\| \sup_{I_i\in \mathcal{S}}\int_{x+B_{t_{i}}}h_n\big\|^{q_0}.
\end{align*}
The first term $J_1$ can be treated as
\begin{align*}
 J_1\leq C_{q_0} 2^{n-k} \frac{1}{ 2^{kd}} \sum_{I_i\in\mathcal{S}_k} \int_{x+(B_{t_{i+1}}\setminus B_{t_i})}
     \|h_n \|^{q_0} \leq C_{q_0} 2^{n-k} \frac{1}{ 2^{kd}} \int_{x+B_{2^{k+1}}} \|h_n \|^{q_0}.
\end{align*}
The second term $J_2$ can be treated as
\begin{align*}
  J_2& \leq C \sum_{I_i\in\mathcal{S}_k} \big(\frac{1}{|B_{t_i}|}-\frac{1}{|B_{t_{i+1}}|}\big)^{q_0}
       \int_{x+B_{2^k}}\big\|h_n\big\|^{q_0}\\
     & \leq C \big(\frac{1}{|B_{2^k}|}-\frac{1}{|B_{2^{k+1}}|}\big)^{q_0}\int_{x+B_{2^k}}\big\|h_n\big\|^{q_0}
       \leq  C_{q_0} 2^{n-k} \frac{1}{ 2^{kd}} \int_{x+B_{2^k}}\big\|h_n\big\|^{q_0}.
\end{align*}
Then we use the similar arguments as the preceding  proof of the long variation. And give
\begin{align*}
  J_1+J_2\leq  C_{q_0}2^{n-k}\lambda^{q_0}A_{2^{k+2}}|d_n|.
\end{align*}
This is just (\ref{3 esti in proof of SV}).
\end{proof}

\subsection{$(L^{\infty}_c,BMO)$ estimates}
Since $\|f\|_{BMO_d}\leq\|g\|_{BMO_d}$ can not be deduced from $f\leq g$, the fact that $\V_q(\A)$ is bounded from $L^{\infty}_c(\B)$ to $BMO_d$ can not be directly deduced from the $(L^{\infty}_c,BMO)$ estimates of  $\mathcal{LV}_{q_0}(\A)$,  $\mathcal{SV}_{q_0}(\A)$ and $\V_{q}(\E)$ through (\ref{vqm by svq0m lvq0m vqe}). However, we will see at the end of this subsection that an argument similar to the proof of (\ref{vqm by svq0m lvq0m vqe}) will be used to deduce the estimate for $\V_q(\A)$ follows from the similar estimates for $\mathcal{LV}_{q_0}(\A)$ and  $\mathcal{SV}_{q_0}(\A)$.

We shall use the equivalent definition of $BMO$ norm (or $BMO_d$ norm),
\begin{align*}
\|g\|_{BMO}\simeq\sup_{Q}\inf_{a_{Q}}\frac{1}{|Q|}\int_{Q}|g-a_{Q}|.
\end{align*}
Let us first prove that $\mathcal{LV}_{q_0}(\A)$ is bounded from $L^{\infty}_c(\R^d;\B)$ to $BMO_d(\R^d)$.
\begin{proof}
Give $f\in L^{\infty}_c(\R^d,\B)$, and a dyadic cube $Q$. We decompose $f$ as
 $f=f\un_{ Q^* }+f\un_{\R^d\setminus Q^*}=f_1+f_2$,
where $Q^*$ is the cube with the same center as $Q$ but three times the side length as the definition in Lemma
\ref{C Z decomposition}. We shall take $a_Q=\mathcal{LV}_{q_0}(\A)f_2(c_Q)$ where $c_Q$ is the center of $Q$. Write $\mathcal{LV}_{q_0}(\A)f-a_Q$ as
$$\mathcal{LV}_{q_0}(\A)f-a_Q=\mathcal{LV}_{q_0}(\A)f-\mathcal{LV}_{q_0}(\A)f_2+\mathcal{LV}_{q_0}(\A)f_2-a_Q,$$
by triangle inequalities,
\begin{align*}
&\frac{1}{|Q|}\int_{Q}|\mathcal{LV}_{q_0}(\A)f -a_{Q}|\\
&\leq\frac{1}{|Q|}\int_{Q}|\mathcal{LV}_{q_0}(\A)f-\mathcal{LV}_{q_0}(\A)f_2|\\
 &\quad            +\frac{1}{|Q|}\int_Q|\mathcal{LV}_{q_0}(\A)f_2(x)- \mathcal{LV}_{q_0}(\A)f_2(c_Q)|dx\\
&\leq\frac{1}{|Q|}\int_{Q}|\mathcal{LV}_{q_0}(\A)(f_1)|\\
&\quad+\frac{1}{|Q|}\int_Q\big(\sum_k\|A_{2^k}f_2(x)-\E_k f_2(x)
            -(A_{2^k}f_2(c_Q)-\E_kf_2(c_Q))\|^{q_0}\big)^{\frac{1}{q_0}}dx\\
&=I+II.
\end{align*}
The first term $I$ is easily estimated by the fact that $\mathcal{LV}_{q_0}(\A)$ is of strong type $(q_0,q_0)$. Indeed,
\begin{align*}
I\leq\big(\frac{1}{|Q|}\int_{Q}|\mathcal{LV}_{q_0}(\A)(f_1)|^{q_0}\big)^{\frac{1}{q_0}}\leq
 C_{q_0}\big(\frac{1}{|Q|}\int_{\R^d}\|f_1\|^{q_0}\big)^{\frac{1}{q_0}}\leq C_{q_0,d}\|\|f\|\|_{\infty}.
\end{align*}
The second term $II$ is controlled by a constant multiple of $\|\|f\|\|_{\infty}$ once we prove that for any $x\in Q$
$$\big(\sum_k\|A_{2^k}f_2(x)-\E_kf_2(x)-(A_{2^k}f_2(c_Q)-\E_kf_2(c_Q))\|^{q_0}\big)^{\frac{1}{q_0}}
  \leq C_{q_0}\|\|f\|\|_{\infty}.$$
If $2^k<\ell(Q)$, then $\E_k f_2$ is supported in $\R^d\setminus Q$ and $B_{2^k}+x$ is contained in $Q^*$ since for any $y\in B_{2^k}$
 $$|x_i+y_i-(c_Q)_i|\leq|x_i -(c_Q)_i|+|y_i|\leq \frac{1}2\ell(Q)+ 2^{k}< \frac32 \ell(Q),$$
where $x_i$ denotes the $i$th element of $x$. Then in this case, we get
$$A_{2^k}f_2(x)-\E_kf_2(x)-(A_{2^k}f_2(c_Q)-\E_kf_2(c_Q))=0,~\mbox{for any}~  x\in Q.$$
Hence it suffices to consider the case $2^k\geq \ell(Q)$. Note that in this case, $Q$ should be contained in some atom of $\sigma_k$, so $\E_kf_2(x)=\E_k f_2(c_Q)$. On the other hand,
\begin{align*}
\|A_{2^k} f_2(x)-A_{2^k} f_2(c_Q)\| &=\frac{1}{|B_{2^k}|} \|\int_{B_{2^k}+x}f_2 -\int_{B_{2^k}+ c_Q}f_2 \|\\
&=\frac{1}{2^{kd}} \|\int_{\R^d}f_2( \un_{ B_{2^k}+ c_Q\setminus  B_{2^k}+x}-\un_{B_{2^k} +c_Q\setminus B_{2^k}+x}) \|\\
&\leq\frac{1}{2^{kd}} \|\|f\|  \|_{\infty} \int_{\R^d}\un_{ B_{2^k}+ c_Q \bigtriangleup  B_{2^k}+x}\\
&\leq\frac{1}{2^{kd}}|  B_{2^k}+ c_Q \bigtriangleup  B_{2^k}+x |\cdot \|\|f\| \|_{\infty}.
\end{align*}
Then the fact that $|B_{2^k}+ c_Q \bigtriangleup B_{2^k}+x |\leq C|x-c_Q|^d\leq C|Q|$ yields
 $$\|A_{2^k} f_2(x)-A_{2^k} f_2(c_Q)\|\leq C|Q|\cdot \|\|f\|\|_{\infty} \frac{1}{2^{kd}}.$$
Finally, the fact that $\ell^{q_0}$ norm is not bigger than $\ell^1$ norm implies
 $$II\leq C|Q|\cdot  \|\|f\|\|_{\infty}\sum_{2^k \geq\ell(Q)}\frac{1}{2^{kd}}\leq C  \|\|f\| \|_{\infty}.$$
\end{proof}

The proof of $(L^{\infty}_c, BMO)$ boundedness of $\mathcal{SV}_{q_0}(\A)$ is a variant of that of $\mathcal{LV}_{q_0}(\A)$, let us give a brief explanation.
\begin{proof}
Give $f\in L^{\infty}_c(\B)$, and a cube $Q$. By an argument similar to that in the proof for $\mathcal{LV}_{q_0}(\A)$, it suffices to prove
\begin{align*}
 \sum_{2^k\geq\ell(Q)}\sum_{2^k<t_i<t_{i+1}\leq 2^{k+1}}\|A_{t_{i+1}} f_2(x)-A_{t_{i}} f_2(x)
  -( A_{t_{i+1}}f_2(c_Q)-A_{t_{i}} f_2(c_Q))\|^{q_0},
\end{align*}
denoted by $I^{q_0}$, is dominated by $\|\|f\|\|^{q_0}_{\infty}$ modulo a constant for any $x\in Q$ and any sequence $(t_i)_i$.
In the short variation case, we further split the sum in $\mathcal{S}_k$ to two parts by comparing $t_{i+1}-t_{i}$ and $\ell(Q)$. We shall denote by $II$ the part in which  $t_{i+1}-t_{i}<\ell(Q)$ , by $III$ the part where $t_{i+1}-t_{i}\geq\ell(Q)$. Hence $I$ is dominated by triangle inequality  by $II+III$. Let us first estimate $II$. By the  triangle inequality,
\begin{align*}
II&\leq \big(\sum_{2^k\geq\ell(Q)}\sum_{\substack{I_i\in \mathcal{S}_k\\t_{i+1}-t_{i}<\ell(Q)}}
   \| A_{t_{i+1}}f_2(x)- A_{t_{i}}f_2(x)\|^{q_0}\big)^{\frac{1}{q_0}}\\
&+\big(\sum_{2^k\geq\ell(Q)}\sum_{\substack{I_i\in \mathcal{S}_k \\ t_{i+1}-t_{i}<\ell(Q)}}
   \|A_{t_{i+1}}f_2(c_Q)- A_{t_{i}}f_2(c_Q)\|^{q_0}\big)^{\frac{1}{q_0}}.
\end{align*}
By the fact that $(B_{t_i})_i$ are nested for $i$, for any $z\in Q$ we have
\begin{align*}
&\|A_{t_{i+1}} f_2(z)- A_{t_{i}} f_2(z)\|\\
&=\big\|\big(\frac{1}{|B_{t_{i+1}}|}-\frac{1}{|B_{t_{i}}|}\big)\int_{B_{t_i}+z}f_2
   +\frac{1}{|B_{t_{i+1}}|}\int_{(B_{t_{i+1}}+z)\setminus(B_{t_{i}}+z)}f_2\big\|\\\
&\leq\big(\frac{1}{|B_{t_{i}}|}-\frac{1}{|B_{t_{i+1}}|}\big)\int_{B_{t_i}+z}\|f_2\|
   +\frac{1}{|B_{t_{i+1}}|}\int_{(B_{t_{i+1}}+z)\setminus(B_{t_{i}}+z)}\|f_2\|\\
&\leq\big(\frac{1}{|B_{t_{i}}|}-\frac{1}{|B_{t_{i+1}}|}\big)|B_{t_i}|\|\|f_2\|\|_{\infty}
   +\frac{1}{|B_{t_{i+1}}|}(|B_{t_{i+1}}|-|B_{t_{i}}|)\|\|f_2\|\|_{\infty}\\
&=2(|B_{t_{i+1}}|-|B_{t_{i}}|)\frac{1}{|B_{t_{i+1}}|}\|\|f\|\|_{\infty}\\
&\leq C\|\|f\|\|_{\infty} \int_{|B_{t_i}|}^{|B_{t_{i+1}}|}\frac{1}{u}du ,
\end{align*}
then by the H\"older inequality and the fact that $t_{i+1}-t_{i}\leq\ell(Q)$ in this case,
\begin{align*}
&\leq C\|\|f\|\|_{\infty} (|B_{t_{i+1}}|-|B_{t_{i}}|)^{1-\frac{1}{q_0}}
       \big(\int_{|B_{t_i}|}^{|B_{t_{i+1}}|}\frac{1}{u^{q_0}}du\big)^{\frac1{q_0}}\\
&\leq C_{d,q_0}\|\|f\|\|_{\infty} \ell(Q)^{1-\frac{1}{q_0}} |{t_{i}}|^{(d-1)(1-\frac{1}{q_0})}
      \big(\int_{|B_{t_i}|}^{|B_{t_{i+1}}|}\frac{1}{u^{q_0}}du\big)^{\frac1{q_0}}.
\end{align*}
Hence
\begin{align*}
II&\leq C_{d,q_0} \ell(Q)^{1-\frac{1}{q_0}} \|\|f\|\|_{\infty}\Big(\sum_{2^k\geq\ell(Q)} 2^{k(d-1)(q_0-1)}
         \sum_{I_i\in \mathcal{S}_k} \int_{|B_{t_i}|}^{|B_{t_{i+1}}|}\frac{1}{u^{q_0}}du \Big)^{\frac{1}{q_0}}\\
  &\leq C_{d,q_0}\ell(Q)^{1-\frac{1}{q_0}} \|\|f\|\|_{\infty}\big(\sum_{2^k\geq\ell(Q)} 2^{k(d-1)(q_0-1)}
         \int^{2^{(k+1)d}}_{ 2^{kd}}\frac{1}{u^{q_0}}du \big)^{\frac{1}{q_0}}\\
  &\leq C_{d,q_0}\ell(Q)^{1-\frac{1}{q_0}} \|\|f\|\|_{\infty}\big(\sum_{2^k\geq\ell(Q)}
         \frac{1}{ 2^{k(q_0-1)}}  \big)^{\frac{1}{q_0}}\\
&\leq C_{d,q_0}\|\|f\|\|_{\infty}.
\end{align*}
We end the proof by the estimate of $III$.  By the  triangle inequality,
\begin{align*}
III& \leq \big(\sum_{2^k\geq\ell(Q)}\sum_{\substack{I_i\in \mathcal{S}_k\\ t_{i+1}-t_i\geq\ell(Q)}}
     \|A_{t_{i+1}} f_2(x)-A_{t_{i+1}} f_2(c_Q)\|^{q_0} \big)^{\frac{1}{q_0}}\\
   & +\big(\sum_{2^k\geq\ell(Q)} \sum_{\substack{I_i\in \mathcal{S}_k\\ t_{i+1}-t_i\geq\ell(Q)}}
     \|A_{t_i}f_2(x)-A_{t_i}f_2(c_Q)\|^{q_0} \big)^{\frac{1}{q_0}}.
\end{align*}
By an argument similar to that in the proof for $\mathcal{LV}_{q_0}(\A)$, for any $B_{t_i}$ we have
\begin{align*}
\|A_{t_i} f_2(x)- A_{t_i} f_2(c_Q)\| \leq C|Q|\|\|f\|\|_{\infty}\frac{1}{|B_{t_i}|}.
\end{align*}
Note that  the interval of $I_i\in\mathcal{S}_k$ such that $t_{i+1}-t_i\geq\ell(Q)$ is smaller $2^k/\ell(Q)$ modulo a constant, hence
\begin{align*}
III& \leq  C|Q|\|\|f\|\|_{\infty} \big(\sum_{2^k\geq\ell(Q)} \sum_{\substack{I_i\in \mathcal{S}_k\\ t_{i+1}-t_i\geq\ell(Q)}}
     \frac{1}{|B_{t_i}|^{q_0}}\big)^{\frac{1}{q_0}}\\
&\leq C_{q_0}|Q|\|\|f\|\|_{\infty} \big(\sum_{2^k\geq\ell(Q)}\frac{2^{k}}{\ell(Q)}\frac{1}{2^{kdq_0}}\big)^{\frac{1}{q_0}}\\
&\leq  C_{q_0}\|\|f\|\|_{\infty}.
\end{align*}
\end{proof}

Now let us explain briefly the proof of the fact that  $\V_q(\A)$ is bounded from $L^{\infty}_c(\B)$ to $BMO_d$.  Give $f\in L^{\infty}_c(\B)$ and a cube $Q$, use the fact that $\ell^{q}$ norm  is not bigger than $\ell^{q_0}$ norm for $q>q_0$, it suffices to prove
\begin{align*}
\big(\sum_{(t_i)_i}\|A_{t_{i+1}} f_2(x)-A_{t_{i}}f_2(x)- (A_{t_{i+1}}f_2(c_Q)-A_{t_{i}} f_2(c_Q))\|^{q_0} \big)^{\frac{1}{q_0}},
\end{align*}
is dominated by $\|\|f\| \|_{\infty}$ modulo a constant for any $x\in Q$ and any regular sequence $(A_{t_i})_i$. As in the proof of (\ref{vqm by svq0m lvq0m vqe}), we divide the intervals $I_i$'s into two sets $\mathcal{S}$ and $\mathcal{L}$. Then the term $\sum_{I_i\in\mathcal{S}}$ is dealt with  in the same way as  that  for $\mathcal{SV}_{q_0}(\A)$.

For $I_i\in\mathcal{L} $, we can rewrite $I_i=(2^{n_i}, 2^{n_{i+1}}]$. Then in the term $\sum_{I_i\in\mathcal{L}}$, we can write
\begin{align*}
&A_{t_{i+1}}f_2(x)-A_{t_{i}} f_2(x)-(A_{t_{i+1}}f_2(c_Q)-A_{t_{i}} f_2(c_Q))\\
&=A_{2^{n_{i+1}}} f_2(x)-\E_{n_{i+1}} f_2(x)-(A_{2^{n_{i+1}}} f_2(c_Q)-\E_{n_{i+1}}f_2(c_Q))\\
&+A_{2^{n_i}} f_2(x)-\E_{n_i}f_2(x)-(A_{2^{n_i}} f_2(c_Q)-\E_{n_i}f_2(c_Q))\\
&+\E_{n_{i+1}}f_2(x)-\E_{n_i}f_2(x)-(\E_{n_{i+1}}f_2(c_Q)-\E_{n_i}f_2(c_Q)),
\end{align*}
where $\E_{t_i}$ is defined to be $2^k<t_i\leq 2^{k+1}$.
Then by triangle inequality, the three associated terms on the right hand side of the previous equality can be dealt with the same way as that for $\mathcal{LV}_{q_0}(\A)$.

The same argument for the proof of Theorem \ref{Zd q variation} works also for vector-valued $q$-variation associated to the family of differential operators on $\mathbb{Z}^d$. For the application to ergodic theory, let us state our results as follows in the case $d=1$. Let $f\in\ell^{1}(\mathbb{Z};\B)$. For any integer $n\geq0$, define
$$A_nf(j)=\frac{1}{n+1}\sum^{n}_{k=0}f(j-k).$$

\begin{corollary}\label{cor:Z1 q variation}
Let $2\leq q_0<q<\infty$ and $\B$ is of martingale cotype $q_0$. Then for any $1<p<\infty$, there exist a constant $C_{p,q}$ such that
\begin{align}\label{Z1 q variation p estimates}
\|\V_q(A)f\|_{p}\leq C_{p,q}\|\|f\|\|_{p},\;\forall f\in \ell^p(\Z;\B);
\end{align}
\end{corollary}

\section{Necessity of Rademacher cotype $q$}

As in the martingale case, one may wonder whether a reverse statement of Theorem \ref{Zd q variation} remains true. At the moment of writing, we do not know how to conclude that $\B$ is of martingale cotype  $q$ from the inequality (\ref{Zd q variation p estimates}). But we can prove the following result.
\begin{theorem}\label{thm: rademacher cotype q}
Let $2\leq q<\infty$ and $\B$ be a Banach space. If there exists a $1<p<\infty$ and a constant $C_{p,q}$ such that
\begin{align}\label{R1 q variation p estimates}
\|\V_q(A)f\|_{p}\leq C_{p,q}\|\|f\|\|_{p},\;\forall f\in L^p(\R;\B),
\end{align}
then $\B$ is of Rademacher cotype $q$.
\end{theorem}
From the proof of Theorem \ref{Zd q variation}, we see that the previous inequality (\ref{R1 q variation p estimates}) holds for all $1<p<\infty$. In particular, we will use the inequality in the case $p=q$. Two transference lemmata are needed for the proof of Theorem \ref{thm: rademacher cotype q}. The first one says that inequality (\ref{R1 q variation p estimates}) remains true if the differential operator $A$ is replaced with Poisson integral $P$.
\begin{lemma}\label{lem: poisson R1}
Let $P(x)=1/{\pi}(1+|x|^2)^{-1}$ and $P_t(x)=1/tP(x/t)$. Define $P_tf(x)=P_t\ast f(x)$. Let $\B$ be a Banach space. If inequality (\ref{R1 q variation p estimates}) holds, then
\begin{align}\label{poisson R1 q variation p estimates}
\|\V_q(P)f\|_{p}\leq C_{p,q}\|\|f\|\|_{p},\;\forall f\in L^p(\R;\B).
\end{align}
\end{lemma}
 The same argument as in Lemma 2.4 of \cite{CJRW00} can be repeated word by word to show Lemma {\ref{lem: poisson R1}}, since $\B$ is a Banach space. We leave the details to the interesting readers.
\begin{remark}
{\rm{(i)}} Actually, Lemma \ref{lem: poisson R1} is true not only for Poisson integrals but for all even functions $\Phi$ satisfying $\lim_{x\rightarrow \infty}\Phi(x)=0$ and $\int_{\R}\Phi'(x)xdx<\infty.$
{\rm (ii)} A higher dimensional version of Lemma \ref{lem: poisson R1} as well as previous {\rm (i)} remains holding, see Lemma 2.4 of \cite{CJRW00}.
\end{remark}

The second lemma says that we can obtain inequality \ref{poisson R1 q variation p estimates}  for Poisson kernel $\mathbb{P}$ on the circle using transference technique developed in \cite{GiTo04}.

\begin{lemma}\label{lem: poisson T1}
Let $\mathbb{P}_t(\theta)=\sum_ne^{-t|n|}e^{2\pi in\theta}$ be the Poisson kernel for the unit disc.  Let $\B$ be a Banach space. If inequality (\ref{R1 q variation p estimates}) holds, then
\begin{align}\label{poisson T1 q variation p estimates}
\|\V_q(\mathbb{P})f\|_{p}\leq C_{p,q}\|\|f\|\|_{p},\;\forall f\in L^p([0,1];\B).
\end{align}
\end{lemma}

\begin{proof}
Given $f\in L^p([0,1];\B)$. For any $x\in \R$, define $T^{x}f(\theta)=f(\theta+x)$. It can be easily checked that $\{T^x:\;x\in\R\}$ is a  strongly continuous one-parameter group of positive invertible linear operators satisfying $\bold{SH}_p$ in Page 135 of \cite{GiTo04}. Then by Remark 2.13 in \cite{GiTo04}, we have
\begin{align}\label{transference result}
\left\|\sup_{(t_i)_i}\left(\sum_{i}\left\|\int_{\R}(P_{t_i}(x)-P_{t_{i+1}}(x))f(\theta-x)dx\right\|^q\right)^{1/q}\right\|_p\leq C_{p,q}\|\|f\|\|_p.
\end{align}
By the Poisson summation formula, for any $t>0$ we have
$$\mathbb{P}_t(x)=\sum_ne^{-t|n|}e^{2\pi inx}=\sum_{n}P_t(x+n).$$ Hence, by the fact that $f(\theta+n)=f(\theta)$ for any $n\in\mathbb{Z}$, we obtain
\begin{align*}
\int_{\R}P_{t}(x)f(\theta-x)dx&=\sum_n\int_{[0,1]+n}P_t(x)f(\theta-x)dx\\
&=\sum_n\int_{[0,1]}P_t(x+n)f(\theta-x-n)dx\\
&=\int_{[0,1]}\sum_nP_t(x+n)f(\theta-x)dx\\
&=\int_{[0,1]}\mathbb{P}_t(x)f(\theta-x)dx=\mathbb{P}_t\ast f(\theta),
\end{align*}
which, together with (\ref{transference result}), implies the desired inequality (\ref{poisson T1 q variation p estimates}).
\end{proof}

Now we are at a position to prove Theorem \ref{thm: rademacher cotype q}. We follow an argument given by Pisier in Proposition 7.5 of \cite{Pis92}.
\begin{proof}
Following the remark after Theorem \ref{thm: rademacher cotype q}, by Lemma \ref{lem: poisson R1} and \ref{lem: poisson T1}, we have for any sequence $(t_i)_i$
\begin{align}\label{intermediate estimate 1}
\sum_i\left\|\|\mathbb{P}_{t_i}\ast f-\mathbb{P}_{t_{i+1}}\ast f\|\right\|^q_q\leq C^q_q\|\|f\|\|_q, \;\forall \;f\in L^q([0,1];\B).
\end{align}
It is easy to check that
$$\|({\mathbb{P}_{t_i}\ast f-\mathbb{P}_{t_{i+1}}\ast f})^\wedge(2^i)\|=|e^{-t_i2^i}-e^{-t_{i+1}2^i}|\|\hat{f}(2^i)\|.$$
On the other hand, by H\"older inequality
\begin{align*}
\|({\mathbb{P}_{t_i}\ast f-\mathbb{P}_{t_{i+1}}\ast f})^\wedge(2^i)\|\leq\left\|\|\mathbb{P}_{t_i}\ast f-\mathbb{P}_{t_{i+1}}\ast f\|\right\|_q.
\end{align*}
Take $t_i=\log(1+1/(2^i-1))$, then
$$|e^{-t_i2^i}-e^{-t_{i+1}2^i}|\rightarrow e^{-1}-e^{-2}>0.$$
Hence we must have
$$\left\|\|\mathbb{P}_{t_i}\ast f-\mathbb{P}_{t_{i+1}}\ast f\|\right\|_q\geq 2^{-1}(e^{-1}-e^{-2})\|\hat{f}(2^i)\|$$
for all $i$ sufficiently large.
Therefore there exists $i_0$ such that for all $(t_{i})_{i\geq i_0}$, (\ref{intermediate estimate 1}) yields
\begin{align}\label{intermediate estimate 2}
\left(\sum_{i\geq i_0}\|\hat{f}(2^i)\|^q\right)^{\frac{1}{q}}\leq C_q\|\|f\|\|_q.
\end{align}
The fact that (\ref{intermediate estimate 2}) implies Rademacher cotype $q$ is well known. Indeed for any fixed sequence of signs $\varepsilon_k=\pm 1$ and any sequence $x_1,\dotsm,x_n$ in $\B$ we deduce from (\ref{intermediate estimate 2}) that
$$\left(\sum_k\|x_k\|^q\right)^{1/q}\leq C_q\left\|\sum_k\varepsilon_ke^{2\pi i2^k\theta}x_k\right\|_{L^q([0,1])}.$$
Taking the $L^q$-norm of both sides with respect to $(\varepsilon_1,\dotsm,\varepsilon_n)$, by Fubini's theorem and  the contraction principle we obtain
\begin{align*}
\left(\sum_k\|x_k\|^q\right)^{1/q}&\leq C_q\left\|\left\|\sum_k\varepsilon_ke^{2\pi i2^k\theta}x_k\right\|_q\right\|_{L^q([0,1])}
\leq \left\|\sum_k\varepsilon_ke^{2\pi i2^k\theta}x_k\right\|_q.
\end{align*}
\end{proof}

\section{Ergodic Averages}
It is well known that many results in harmonic analysis on $\mathbb{Z}$ can be transferred to the ergodic theory. In this section, we will show that the vector-valued $q$-variation for the ergodic averages is also bounded on $L^p(\B)$ with $1<p<\infty$.  Let $(\Omega_1,\mu)$ and $(\Omega_2,\mu)$ be two measure spaces.  An operator $T: L^p(\Omega_1)\rightarrow L^p(\Omega_2)$ is called contractively regular if the following inequality holds
$$\left\|\sup_{k\geq1}|T(x_k)|\right\|_{L^p(\Omega_1)}\leq \left\|\sup_{k\geq1}|x_k|\right\|_{L^p(\Omega_2)}.$$
for any finite sequence $(x_k)_{k\geq1}$ in $L^p(\Omega_1)$. Any contractively regular operator $T$ can be extended to a contractive operator on the Bochner space $L^p(\Omega_1; E)$ for any Banach space $E$, i.e.
\begin{align}\label{extension of T}
\|T\otimes I_{E}: L^p(\Omega_1;E)\rightarrow L^p(\Omega_2;E)\|\leq1.
\end{align}
In this paper, any extension of any operator $S$ will be still denoted by $S$ when no confusion occurs.

Obviously, positive contractions are regular. On the other hand, it is easy to check that if $T$ is a contraction on $L^1(\Omega_1)$ and $L^{\infty}(\Omega_1)$, then $T$ is contractively regular on  $L^p(\Omega_1)$. We refer to \cite{Mey91}, \cite{Pel76} and \cite{Pis94} for more details and complements.

Let $T=(T_t)_{t>0}$ be a bounded strong continuous semigroup. For any $t>0$, the ergodic averages of $T$ is defined as
$$M_t(T)=\frac{1}{t}\int^{t}_{0}T_rdr.$$
The main result of this section is the following theorem.

\begin{theorem}\label{thm:ergodic average continuous}
Let $1<p<\infty$. Let $T=(T_t)_{t>0}$ be a strong continuous semigroup on $L^p(\Omega)$  and every $T_t$ is contractively regular. Let  $2\leq q_0<q<\infty$ and $\B$ be of martingale cotype $q_0$.  Let $\B$ be an UMD lattice having Fatou property. Then  there exist a constant $C_{p,q}$ such that
\begin{align}\label{ergodic average continuous}
\left\|\omega\rightarrow\|((M_t(T)f)(\omega))_{t\geq0}\|_{V_q(\B)}\right\|_p\leq C_{p,q}\|\|f\|\|_p.
\end{align}
\end{theorem}

This proof is based on the transference principle, together with Theorem \ref{Zd q variation}. 

\begin{proof}
By Fendler's dilation theorem \cite{Fen97}, there exist another measure space $(\hat{\Omega},\hat{\mu})$, a strongly continuous group $U={U_t}_{t\in\mathbb{R}}$ of regular isometries on $Lp(\hat{\Omega})$, a positive isometric embedding $Q$ from $L_p(\Omega)$ into $L_p(\hat{\Omega})$ and a regular projection $J$ from $L_p(\hat{\Omega})$ onto $L_p(\Omega)$ such that
$T_t=QU_tJ$, for any $t>0$.

Given $f\in L^p(\Omega;\B)$. 
It is easy to check that $$\|(M_t(T)f)_{t>0}\|_{L^p(\B(V_q))}=\lim_{a\rightarrow\infty }\|(M_t(T)f)_{0<t<a}\|_{L^p(\B(V_q))}.$$
Now fix an integer $a>0$. For any $t>0$, we clearly have
$$M_t(T)=QM_t(U)J.$$
Since $\|Q\|_r\leq1$, it follows from (\ref{extension of T}) that
\begin{align}\label{intermediate estimate}
\|(M_t(T)f)_{0<t<a}\|_{L^p(\Omega;V_q(\B))}\leq\|(M_t(U)J(f))_{0<t<a}\|_{L^p(\hat{\Omega};V_q(\B))}.
\end{align}
Using the regularity of $U$, we have 
\begin{align*}
\|(M_t(U)J(f))_{0<t<a}\|_{L^p(\hat{\Omega};V_q(\B))}&=\|(M_t(U)U^{-s}U^{s}J(f))_{0<t<a}\|_{L^p(\hat{\Omega};V_q(\B))}\\
&\leq\|(M_t(U)U^{s}J(f))_{0<t<a}\|_{L^p(\hat{\Omega};V_q(\B))}
\end{align*}
for any $s>0$. 
Hence for any $b>0$, we have
\begin{align*}
&\|(M_t(U)J(f))_{0<t<a}\|^p_{L^p(\hat{\Omega};V_q(\B))}\\
&\leq \frac{1}{b}\int^b_{0}\|(M_t(U)U^{s}J(f))_{0<t<a}\|^p_{L^p(\hat{\Omega};V_q(\B))}\\
&=\frac{1}{b}\int_{\hat{\Omega}}\int^b_{0}\|(M_t(U)U^{s}J(f))_{0<t<a}\|^p_{V_q(\B))}d\hat{\omega}.
\end{align*}
Now we define a $\B$-valued function on $\mathbb{R}\times \hat{\Omega}$, $g(t,\hat{\omega})=\chi_{0, a+b}(t)U^{t}J(f)(\hat{\omega})$. Then
\begin{align*}
M_t(U)U^{s}J(f)(\hat{\omega})&=\frac{1}{t}\int^{t}_{0}U^{r+s}J(f)(\hat{\omega})=A_t(g(\cdot,\hat{\omega}))(s).
\end{align*}
Now apply Theorem \ref{Zd q variation},  using the regularity of $U$ and $J$, we have
\begin{align*}
&\|(M_t(U)J(f))_{0<t<a}\|^p_{L^p(\hat{\Omega};V_q(\B))}\leq C^p_{p,q}\frac{1}{b}\int_{\hat{\Omega}}\int^{a+b}_{0}\|g(s,\hat{\omega})\|_{\B}^p d\hat{\omega}\\
&\leq \leq C^p_{p,q}\frac{1}{b}\int_{\hat{\Omega}}\int^{a+b}_{0}\|J(f)(\hat{\omega})\|_{\B}^p d\hat{\omega}\leq\frac{a+b}{b}C^p_{p,q}\|f\|^p_{L^p(\B)}.
\end{align*}
\end{proof}

A discret version of Theorem \ref{thm:ergodic average continuous} is also true by the same transference techniques together with Corollary \ref{cor:Z1 q variation}.  Let $T: L^p(\Omega)\rightarrow L^p(\Omega)$ be a contractively regular operator. For any integer $n\geq0$, the ergodic averages of $T$ is defined as
$$M_n(T)=\frac{1}{n+1}\sum^{n}_{k=0}T^k.$$

\begin{corollary}\label{cor:ergodic average discret}
Let $1<p<\infty$ and $2\leq q_0<q<\infty$. Let $\B$ be of martingale cotype $q_0$. Then  there exist a constant $C_{p,q}$ such that
\begin{align}\label{ergodic average discrete}
\|(M_n(T)f)_{n\geq0}\|_{L^p(v_q(\B))}\|\leq C_{p,q}\|\|f\|\|_{p},\;\forall f\in L^p(\Omega;\B).
\end{align}
\end{corollary}

\section{Symmetric diffusion semigroups}

The aim of this section is to show that the result in Theorem \ref{thm:ergodic average continuous} is still true for the symmetric semigroups themselves instead of the ergodic averages. We need the following lemma, which is a vector-valued version of Lemma 2.2 in \cite{LeXu2}. We omit its proof.

\begin{lemma}\label{lem: approximation properties 2}
Let $(f_t)_{t\geq0}$ be a family of $L^p(\Omega; E)$ and assume that:
\begin{enumerate}[\rm(i)]
\item For a.e. $\omega\in\Omega$, the function $t\rightarrow f_t(\omega)$ is continuous in $E$ on $(0,\infty)$;
\item There exists a constant $C>0$ such that whenever $t_0<t_1<\dotsm<t_m$ is a finite increasing sequence of positive real numbers, we have
$$\|(f_{t_0},f_{t_1},\dotsm,f_{t_m})\|_{L^p(\Omega; v^m_q(E))}\leq C.$$
\end{enumerate}
Then $(f_t(\omega))_{t>0}$ belongs to $V_q(E)$ for a.e. $\omega\in\Omega$, and the resulting function belongs to $L^p(\Omega)$, i.e.
$$\left\|\omega\rightarrow\|(f_t(\omega))_{t>0}\|_{V_q(E)}\right\|_{p}\leq C.$$
\end{lemma}

The main result is stated as follows.
\begin{theorem}\label{thm: analytic semigroup}
Let $(T_t)_{t\geq0}$ be a symmetric diffusion semigroup. Let $2\leq q_0<\infty$ and $\B$ be a Banach space of martingale cotype $q_0$ which is an interpolation space between a Hilbert space and another Banach space $B_0$ of martingale cotype $q_1$ with $2\leq q_1<\infty$. Then for any $q_0< q<\infty$, any integer $m\geq0$ and any $f\in L^p(\Omega;\B)$,  the family $\big(t^m\frac{\partial^m}{\partial t^m}(T_t(f))\big)_{t>0}$ belongs to $V_q(\B)$ for a.e. $\omega\in\Omega$ and we have an estimate
\begin{align}\label{analytic semigroup continuous}
\left\|\|\omega\rightarrow\big(t^m\frac{\partial^m}{\partial t^m}(T_t(f))\big)_{t>0}\|_{V_q(\B)}\right\|_{p}\lesssim \|\|f\|\|_p,\;\forall f\in L^p(\Omega;\B).
\end{align}
In particular, when $m=0$, the family $(T_t(f))_{t>0}$ belongs to $V_q(\B)$ for a.e. $\omega\in\Omega$ and we have an estimate
\begin{align*}
\left\|\|\omega\rightarrow (T_t(f))_{t>0}\|_{V_q(\B)}\right\|_{p}\lesssim \|\|f\|\|_p,\;\forall f\in L^p(\Omega;\B).
\end{align*}
\end{theorem}
\begin{remark}
When $\B$ is a Banach lattice of martingale cotype $q_0$, from the argument after Problem 2 in \cite{MTX06}, there exists another Banach lattice $B_0$ of martingale cotype $q_1$ with $q_1\geq2$. On the other hand, noncommutative $L_{q_0}$ spaces are also Banach spaces satisfying the interpolation property.
\end{remark}
As said previously, we  first prove a discrete version, then the continuous case will be deduced from the discrete case by approximation. Let $T: L^p(\Omega)\rightarrow L^p(\Omega)$ and let
$$\Delta^m_n\equiv\Delta^m_n(T)=T^n(T-I)^m$$
for any integers $n,m\geq0$. Note that $(\Delta^m_n)_{n\geq0}$ is the $m$-difference sequence of $(T^n)_{n\geq0}$.

\begin{proposition}\label{pro: analytic semigroup}
Let $2\leq q_0<\infty$ and $\B$ be a Banach space of martingale cotype $q_0$. Assume T is a contractively regular operator on $L^p(\Omega)$ such that for any integer $m\geq0$ and any $f\in L^p(\Omega;\B)$,
\begin{align}\label{littlewood-paley}
\left\|\left(\sum^{\infty}_{n=0}\frac{1}{n+1}\|(n+1)^{m+1}\Delta^{m+1}_n(f)\|^{q_0}\right)^{\frac{1}{q_0}}\right\|_{p}\leq C_{m,q_0}\|\|f\|\|_p.
\end{align}
 Then we have
\begin{align}\label{analytic semigroup}
\|(n^m\Delta^m_n(f))_{n\geq1}\|_{L^p(v_q(\B))}\lesssim \|\|f\|\|_p,\;\forall f\in L^p(\Omega;\B).
\end{align}
 for any integer $m\geq0$.
In particular, when $m=0$, we get the estimate for the discrete semigroup
$$\|(T^n(f))_{n\geq1}\|_{L^p(v_q(\B))}\lesssim \|\|f\|\|_p,\;\forall f\in L^p(\Omega;\B).$$
\end{proposition}
This proposition will be shown later. Let us first prove Theorem \ref{thm: analytic semigroup}, which requires  the fractional averages on the power of $T$ (see Page 409 in \cite{JuXu07}). Given a complex number $\alpha$ and a nonnegative integer n, set
$$A^{\alpha}_n=\frac{(\alpha+1)(\alpha+2)\dotsm(\alpha+n)}{n!}$$
and
$$S^{\alpha}_n\equiv S^{\alpha}_n(T)=\sum^n_{k=0}A^{\alpha-1}_{n-k}T^k,\;M^{\alpha}_n\equiv M_n^{\alpha}(T)=(n+1)^{-\alpha}S^{\alpha}_n.$$
The $(M^{\alpha}_n)_{n\geq0}$ are the fractional averages of $(T^n)_{n\geq0}$. Note that $M^0_n=T^n$ and $M^1_n$ is the usual ergodic average $M_n$ already considered before. Also if $\alpha$ is a negative integer $-m$, then  $S^{-m}_n=\Delta^m_n$. Thus $M^{-m}_n=(n+1)^m\Delta^m_n$.

Now we are ready to show Theorem \ref{thm: analytic semigroup}.
\begin{proof}
Let $m\geq0$ be an integer. By the comments after (\ref{extension of T}), we have $T_t$ extends contractively on all $L^p(\Omega;\B)$ with $1\leq p\leq\infty$ for any Banach spaces and any $t>0$. Then by the density argument, it follows from the lemma in Page 72 of \cite{Ste70} that the function
$$t\rightarrow t^m\frac{\partial^m}{\partial t^m}(T_t(f))(\omega)$$
is continuous in $\B$ for a.e. $\omega\in\Omega$.  To finish the proof, it suffices to show that  $T_t$ satisfy the estimate (\ref{analytic semigroup}) uniformly. Indeed, then using an approximation argument as in the proof of Corollary 4.2 in \cite{LeXu2}, we deduce that for any $0<t_0<t_1<\dotsm<t_m$, we have
$$\|(T_{t_0}(f), T_{t_1}(f),\dotsm, T_{t_m}(f))\|_{L^p(v^m_q(\B))}\leq C\|\|f\|\|_p.$$
The desired result then follows from Lemma \ref{lem: approximation properties 2}.

Following  from Proposition \ref{pro: analytic semigroup}, it suffices to prove  that  $T_t$ satisfy the estimate (\ref{littlewood-paley}) uniformly which will be deduced by complex interpolation from the following two estimates
\begin{align}\label{estimate1}
\left\|\left(\sum^{\infty}_{n=0}\frac{1}{n+1}\|(n+1)\Delta_nM^{\alpha}_n(f)\|_{B_0}^{q_1}\right)^{\frac{1}{q_1}}\right\|_{p}\leq C_{m,q_1}\|\|f\|\|_p
\end{align}
for any $\alpha\in\mathbb{C}$ with $Re(\alpha)>1$ and $f\in L^p(\Omega;B_0)$, and
\begin{align}\label{estimate2}
\left\|\left(\sum^{\infty}_{n=0}\frac{1}{n+1}\|(n+1)\Delta_nM^{\alpha}_n(f)\|_{H}^{2}\right)^{\frac{1}{2}}\right\|_{p}\leq C_{m}\|\|f\|\|_p
\end{align}
for any $\alpha\in\mathbb{C}$ and $f\in L^p(\Omega;H)$.

As in \cite{Ste70}, the estimate (\ref{estimate1}) follows from the same inequality in the case $\alpha=1$, which is a further consequence of Lemma 2.4 in \cite{MTX06}, since $T_t$ is the square of $T_{t/2}$ and hence admits Rota's dilation. And the estimate (\ref{estimate2}) can be deduced again by complex interpolation from the following two estimates
\begin{align}\label{estimate21}
\left\|\left(\sum^{\infty}_{n=0}\frac{1}{n+1}\|(n+1)\Delta_nM^{\alpha}_n(f)\|_{H}^{2}\right)^{\frac{1}{2}}\right\|_{2}\leq C_{m}\|\|f\|\|_2
\end{align}
for any $\alpha\in\mathbb{C}$ and $f\in L^p(\Omega;H)$, and
\begin{align}\label{estimate22}
\left\|\left(\sum^{\infty}_{n=0}\frac{1}{n+1}\|(n+1)\Delta_nM^{\alpha}_n(f)\|_{H}^{2}\right)^{\frac{1}{2}}\right\|_{p}\leq C_{m}\|\|f\|\|_p
\end{align}
for any $\alpha\in\mathbb{C}$ with $Re(\alpha)>1$ and $f\in L^p(\Omega;H)$ for all $1<p<\infty$.

Finally, estimate (\ref{estimate21}) follows from spectral decomposition as in \cite{Ste70} and estimate (\ref{estimate22}) follows again from Rota's dilation since $T_t$ is the square of $T_{t/2}$.
\end{proof}

Now we are at a position to prove Proposition \ref{pro: analytic semigroup} following the pattern set up in \cite{LeXu2}, we need the following two elementary estimates whose proof is left for readers.

\begin{lemma}\label{lem: elementary estimate 1}
For any integer $m\geq0$ and $2\leq q_0<\infty$, there exists a constant $K_{m,q_0}$ such that for any $n\geq1$,
$$\left(\sum^{2n}_{j=n}(j+1)^{\frac{1-q_0m}{q_0-1}}\right)^{\frac{q_0-1}{q_0}}\leq K_{m,q_0} n^{1-m}.$$
\end{lemma}

\begin{lemma}\label{lem: elementary estimate 2}
For any sequences $(\delta_n)_{n\geq0}\in v_1$ and $(z_n)_{n\geq0}\in L^p(\Omega; v_q(\B))$, we have $(\delta_nz_n)_{n\geq0}\in L^p(\Omega;v_q(\B))$ and
$$\|(\delta_nz_n)_{n\geq0}\|_{L^p(v_q(\B))}\leq3\|(\delta_n)_{n\geq0}\|_{v_1}\|(z_n)_{n\geq0}\|_{L^p(v_q(\B))}.$$
\end{lemma}

\begin{proof}
We define
$$\Delta^{-1}_n=nM_{n-1}(T)=\sum^{n-1}_{j=0}T^j,\;\forall n\geq1.$$
The following decomposition formula for any $m\geq0$, which has been shown in (4.10) in \cite{LeXu2}, plays an important role in the induction argument.
\begin{align*}\label{decomposition formula}
n^{m}\Delta^m_{2n+1}&=n^{m-1}\sum^{2n}_{j=n}(j+1)\Delta^{m+1}_j-n^{m-1}(n+1)(\Delta^m_{2n+1}-\Delta^m_n)\\
&\;\;\;\;+n^{m-1}\Delta^{m-1}_{2n+1}-n^{m-1}\Delta^{m-1}_{n+1}\\
&=A_n-\frac{n+1}{n}B_n+n^{m-1}\Delta^{m-1}_{2n+1}-n^{m-1}\Delta^{m-1}_{n+1}.
\end{align*}
Fix $f\in L^p(\Omega;\B)$. To finish the proof, it suffices to prove
\begin{equation}\label{estimate of An}
\|(A_n(f))_{n\geq0}\|_{L^p(v_q(\B))}\lesssim\|\|f\|\|_p
\end{equation}
and
\begin{align}\label{estimate of Bn}
\|(B_n(f))_{n\geq0}\|_{L^p(v_q(\B))}\lesssim\|\|f\|\|_p.
\end{align}
Indeed, the two estimates, together with Lemma \ref{lem: elementary estimate 2}, imply
$$\|((n^m\Delta^m_{2n+1})(f))_{n\geq0}\|_{L^p(v_q(\B))}\lesssim\|\|f\|\|_p$$
since 
\begin{align*}
\|(\frac{n+1}{n}B_n(f))_{n\geq1}\|_{L^p(v_q(\B))}\lesssim\|\|f\|\|_p,
\end{align*}
\begin{align*}
\|(n^{m-1}\Delta^{m-1}_{n+1}(f))_{n\geq1}\|_{L^p(v_q(\B))}\lesssim\|\|f\|\|_p
\end{align*}
and
\begin{align*}
\|(n^{m-1}\Delta^{m-1}_{2n+1}(f))_{n\geq1}\|_{L^p(v_q(\B))}\lesssim\|\|f\|\|_p,
\end{align*}
whence further yields the desired result by using again (\ref{estimate of Bn}) and the identity $n\Delta^m_n=n^m\Delta^m_{2n+1}-B_n$.

As in \cite{LeXu2}, we shall prove
the estimates of (\ref{estimate of An}) and (\ref{estimate of Bn}) by deducing the pointwise estimates
$$\|(A_n(f))_{n\geq0}\|_{v_q(\B)}, \;\|(B_n(f))_{n\geq0}\|_{v_q(\B)}\lesssim \Phi_{m,q_0}(f)$$
where $\Phi_{m,q_0}(f)$ denotes the vector-valued Littlewood-Paley function
$$\Phi_{m,q_0}(f)=\left(\sum^{\infty}_{j=1}\frac{1}{j+1}\|(j+1)^{m+1}\Delta^{m+1}_j(f)\|^{q_0}\right)^{\frac{1}{q_0}}$$
which belongs to $L^p(\Omega)$ by the assumption (\ref{littlewood-paley}).

We only deal with  the sequence $(A_n)_{n\geq0}$, since similar argument works also for the sequence $(B_n)_{n\geq1}$ (see e.g. the paragraph after (4.14) in \cite{LeXu2}). Given a fixed increasing sequence of integers $(n_k)_{k\geq0}$ with $n_0=1$. For any $k\geq1$, we define
\begin{align*}
a_k=\left\{\begin{array}{cc}n^{m-1}_k\sum^{2n_k}_{j=2n_{k-1}+1}(j+1)\Delta^{m+1}_j& \;\mathrm{if}\; 2n_{k-1}\geq n_k\\
       n^{m-1}_k\sum^{2n_k}_{j=n_k}(j+1)\Delta^{m+1}_j& \;\mathrm{if}\; 2n_{k-1}< n_k\end{array}\right.
\end{align*}
$$b_k=\left\{\begin{array}{cc}-n^{m-1}_{k-1}\sum^{n_k-1}_{j=n_{k-1}}(j+1)\Delta^{m+1}_j& \;\mathrm{if}\; 2n_{k-1}\geq n_k\\
       -n^{m-1}_k\sum^{2n_{k-1}}_{j=n_{k-1}}(j+1)\Delta^{m+1}_j& \;\mathrm{if}\; 2n_{k-1}< n_k\end{array}\right.$$
       and
$$c_k=\left\{\begin{array}{cc}(n^{m-1}_k-n^{m-1}_{k-1})\sum^{2n_{k-1}}_{j=n_k}(j+1)\Delta^{m+1}_j& \;\mathrm{if}\; 2n_{k-1}\geq n_k\\
       0& \;\mathrm{if}\; 2n_{k-1}< n_k\end{array}\right.$$
       which yields a decomposition
\begin{align}\label{identity 2}
A_{n_k}-A_{n_{k-1}}=a_k+b_k+c_k.
\end{align}
Let $\theta=1/{q_0}-m$ such that $(1-\theta)q_0=q_0(m+1)-1$ and $\theta q_0/(q_0-1)=(1-q_0m)/(q_0-1)$. Using H\"older inequality, if $2n_{k-1}\geq n_k$, we have
\begin{align*}
\|a_k(f)\|&\leq n_k^{m-1}\sum^{2n_k}_{j=2n_{k-1}+1}(j+1)\|\Delta^{m+1}_j(f)\|\\
&\leq  n_k^{m-1}\left(\sum^{2n_k}_{j=2n_{k-1}+1}(j+1)^{(1-\theta)q_0}\|\Delta^{m+1}_j(f)\|^{q_0}\right)^{\frac{1}{q_0}}\\
&\;\;\;\;\;\;\;\;\;\;\;\;\;\;\;\;\;\;\;\;\;\;\;\;\cdot
\left(\sum^{2n_k}_{j=2n_{k-1}+1}(j+1)^{\frac{\theta q_0}{q_0-1}}\right)^{\frac{q_0-1}{q_0}}\\
&=n_k^{m-1}\left(\sum^{2n_k}_{j=2n_{k-1}+1}(j+1)^{q_0(m+1)-1}\|\Delta^{m+1}_j(f)\|^{q_0}\right)^{\frac{1}{q_0}}\\
&\;\;\;\;\;\;\;\;\;\;\;\;\;\;\;\;\;\;\;\;\;\;\;\;\cdot
\left(\sum^{2n_k}_{j=2n_{k-1}+1}(j+1)^{\frac{1-q_0m}{q_0-1}}\right)^{\frac{q_0-1}{q_0}}.
\end{align*}
Similarly if $2n_{k-1}<n_k$, we have
\begin{align*}
\|a_k(f)\|&\leq  n_k^{m-1}\left(\sum^{2n_k}_{j=n_k}(j+1)^{(1-\theta)q_0}\|\Delta^{m+1}_j(f)\|^{q_0}\right)^{\frac{1}{q_0}}\\
&\;\;\;\;\;\;\;\;\;\;\;\;\;\;\;\;\;\;\;\;\;\;\;\;\cdot
\left(\sum^{2n_k}_{j=n_k}(j+1)^{\frac{\theta q_0}{q_0-1}}\right)^{\frac{q_0-1}{q_0}}\\
&=n_k^{m-1}\left(\sum^{2n_k}_{j=2n_{k-1}+1}(j+1)^{q_0(m+1)-1}\|\Delta^{m+1}_j(f)\|^{q_0}\right)^{\frac{1}{q_0}}\\
&\;\;\;\;\;\;\;\;\;\;\;\;\;\;\;\;\;\;\;\;\;\;\;\;\cdot
\left(\sum^{2n_k}_{j=n_k}(j+1)^{\frac{1-q_0m}{q_0-1}}\right)^{\frac{q_0-1}{q_0}}.
\end{align*}
Thus by Lemma \ref{lem: elementary estimate 1}, we have in both case
$$\|a_{k}(f)\|^{q_0}\leq K^{q_0}_{m,q_0}\sum^{2n_k}_{j=2n_{k-1}+1}(j+1)^{q_0(m+1)-1}\|\Delta^{m+1}_j(f)\|^{q_0}.$$
Summing up, we deduce that
\begin{align}\label{estimate of a}
\sum^{\infty}_{k=1}\|a_{k}(f)\|^{q_0}\leq K^{q_0}_{m,q_0}\Phi_{m,q_0}(f)^{q_0}.
\end{align}
By a similar deduction, we also have
\begin{align}\label{estimate of b}
\sum^{\infty}_{k=1}\|b_{k}(f)\|^{q_0}\leq K^{q_0}_{m,q_0}\Phi_{m,q_0}(f)^{q_0}.
\end{align}

Now, let us turn to the last but the most involved term $c_k(f)$. It suffices to deal with the case $2n_{k-1}\geq n_k$. Using H\"older inequality and Lemma \ref{lem: elementary estimate 1}, we have
\begin{align*}
\|c_k(f)\|^{q_0}&\leq |n_k^{m-1}-n_{k-1}^{m-1}|^{q_0}\sum^{2n_{k-1}}_{j=n_k}(j+1)^{q_0(m+1)-1}\|\Delta^{m+1}_j(f)\|^{q_0}\\
&\;\;\;\;\;\;\;\;\;\;\;\;\;\;\;\;\;\;\;\;\;\;\;\;\cdot
\left(\sum^{2n_{k-1}}_{j=n_k}(j+1)^{\frac{1-q_0m}{q_0-1}}\right)^{{q_0-1}}\\
&\leq K^{q_0}_{m,q_0}\left(\frac{n_k^{m-1}-n_{k-1}^{m-1}}{n_k^{m-1}}\right)^{q_0}\sum^{2n_{k-1}}_{j=n_k}(j+1)^{q_0(m+1)-1}\|\Delta^{m+1}_j(f)\|^{q_0}.
\end{align*}
For any integer $j\geq1,$ define
$$J_j=\{k\geq1:n_k\leq j\leq2n_{k-1}\},$$
and set
$$\Lambda_j=\sum_{k\in J_j}\left(\frac{n_k^{m-1}-n_{k-1}^{m-1}}{n_k^{m-1}}\right)^{q_0}.$$
Then by Fubini theorem,
\begin{align*}
\sum^{\infty}_{k=1}\|c_k(f)\|^{q_0}&\leq K^{q_0}_{m,q_0}\sum^{\infty}_{j=1}\Lambda_j (j+1)^{q_0(m+1)-1}\|\Delta^{m+1}_j(f)\|^{q_0}.
\end{align*}
Let us now estimate the $\Lambda_j$'s. Observe that if $J_j$ is a non empty set, then it is a finite interval of integers. Thus it reads as
$$J_j=\{k_j-N+1,k_j-N+2,\dotsm, k_j-1,k_j\}$$
where $k_j$ is the biggest element of $J_j$ and $N$ is the integer such that $k_j-N+1$ is the least integer  such that $j\leq 2n_{k_j-N}$.

Suppose that $m\geq2$, then the sequence $(n^{m-1}_k)_{k}$ is increasing. Thus
\begin{align*}
\sum_{k\in J_j}n^{m-1}_k-n^{m-1}_{k-1}=\sum^{N-1}_{r=0}n^{m-1}_{k_j-r}-n^{m-1}_{k_j-r-1}=n^{m-1}_{k_j}-n^{m-1}_{k_j-N}\leq n^{m-1}_{k_j},
\end{align*}
which is further smaller then $j^{m-1}$, since $k_j\in J_j$ hence $n_{k_j}\leq j.$
On the other hand, we have $j\leq 2n_k$ for any $k\in J_j$, hence
$$\sum_{k\in J_j}\frac{n_k^{m-1}-n_{k-1}^{m-1}}{n_k^{m-1}}\leq \left(\frac{2}{j}\right)^{{m-1}}
\sum_{k\in J_j}n_k^{m-1}-n_{k-1}^{m-1}\leq 2^{m-1}.$$
We immediately deduce that $\Lambda_j\leq 2^{(m-1)q_0}$. In the case $m=0$, we have similarly
$$\sum_{k\in J_j}\frac{n_k^{-1}-n_{k-1}^{-1}}{n_k^{-1}}\leq j\sum_{k\in J_j}n_k^{-1}-n_{k-1}^{-1}\leq jn^{-1}_{k_j-N}\leq 2.$$
So in this case $\Lambda_j\leq 2^{q_0}$. Finally $\Lambda_j=0$ when $m=1$.

This shows that in all cases, we have
$$\sum^{\infty}_{k=1}\|c_{k}(f)\|^{q_0}\leq K'^{q_0}_{m,q_0}\Phi_{m,q_0}(f)^{q_0}.$$

By the identity (\ref{identity 2}), this estimate together with (\ref{estimate of a}) and (\ref{estimate of b}) yields
\begin{align*}
\sum^{\infty}_{k=1}\|A_{n_k}(f)-A_{n_{k-1}}(f)\|^{q_0}\leq (6K^{q_0}_{m,q_0}=K'^{q_0}_{m,q_0})\Phi_{m,q_0}(f)^{q_0},
\end{align*}
which yields the desired estimates since the upper bound does not depend on the sequence $(n_k)_{k\geq0}$ and $v_q(\B)\subset v_{q_0}(\B)$ when $q>q_0\geq2$.
\end{proof}

\section{Applications}
In this section, we are concerned with the properties of the convergence in connection with variational inequalities. Restriceted to Banach spaces of martingale cotype $q_0$ with $2\leq q_0<\infty$, we will see that our results improve the ones by Cowling/Leinert and Taggart.

Let $T=(T_t)_{t\geq0}$ be a symmetric diffusion semigroup.  By definiton,  $T_t$ is contractive on all $L^p$ with $1\leq p\leq\infty$ for any $t>0$, hence $T_t$ is contractively regular on all $L^p$ with $1\leq p\leq\infty$ for any $t>0$, which implies by (\ref{extension of T}) that $T_t$ extends contractively on all $L^p(\Omega;\B)$ with $1\leq p\leq\infty$ for any Banach spaces and any $t>0$. On the other hand, by definition, for any $f\in L^p(\Omega)$, $T_tf\rightarrow f$ in the $L^p$-norm as $t\rightarrow0^+$. Therefore by the density argument, for any $f\in L^p(\Omega;\B)$, $T_tf\rightarrow f$ in the $L^p(B)$-norm as $t\rightarrow0^+$.

Let $A$ denote the infinitesimal generator of $T$. By mean ergodic theorem, we have direct sum decomposition
$$L^p(\Omega)=N(A)\oplus \overline{R(A)}.$$
If we let $P_A: L^p(\Omega)\rightarrow L^p(\Omega)$ denotes the corresponding projection onto $N(A)$, then
$$T_tf\rightarrow P_A(f),\;\;\mathrm{as}\;t\rightarrow\infty$$
in $L^p$-norm for any $f\in L^p(\Omega)$. By by the density argument, this limit holds also true for any $f\in L^p(\Omega;\B)$.

Now applying Theorem \ref{thm: analytic semigroup}, we deduce the following individual ergodic theorems.
\begin{corollary}
Let $T=(T_t)_{t\geq0}$ be a symmetric diffusion semigroup. Let $2\leq q_0<\infty$ and $\B$ be a Banach space of martingale cotype $q_0$ which is an interpolation space between a Hilbert space and another Banach space $\B_0$ of martingale cotype $q_1$ with $q_0<q_1<\infty$. Let $f\in L^p(\Omega;\B)$ with $1<p<\infty$, then for almost every $\omega\in\Omega$,
$$[T_tf](\omega)\rightarrow [P_A(f)](\omega) \;\mathrm{in}\; B\mathrm{-norm},\;\;\mathrm{as}\;t\rightarrow\infty$$
and
$$[T_tf](\omega)\rightarrow f(\omega) \;\mathrm{in}\; B\mathrm{-norm},\;\;\mathrm{as}\;t\rightarrow\infty.$$
\end{corollary}

As said in the introduction of this paper, the pointwise convergence has been obtained by Cowling and Leinert for all Banach spaces $\B$. However, Theorem \ref{thm: analytic semigroup} provides us quantitvative information on the rate of the convergence, which depends on the geometric property of the Banach space under consideration.  This requires the notion of $\lambda$-jump functions. For any $\lambda>0$ and any family $a=(a_t)_{t\geq0}$ in Banach space $\B$. One defines $N(a,\lambda)$ to be the supremum of all integers $N\geq0$ for which there is an increasing sequence
$$0<s_1<t_1\leq s_2<t_2\leq\dotsm\leq s_N<t_N$$
so that $\|a_{t_k}-a_{s_k}\|>\lambda$ for each $k=1,\dotsm,N$. It is clear that for any $1\leq q<\infty$,
$$\lambda^qN(a,\lambda)\leq \|a\|^q_{V_q(B)}.$$

By Theorem \ref{thm: analytic semigroup}, we immediately obtain the following jump estiamtes.
\begin{corollary}
Let $T=(T_t)_{t\geq0}$ be a symmetric diffusion semigroup. Let $2\leq q_0<\infty$ and $\B$ be a Banach space of martingale cotype $q_0$ which is an interpolation space between a Hilbert space and another Banach space $B_0$ of martingale cotype $q_1$ with $2\leq q_1<\infty$.  For any $f\in L^p(\Omega;B)$ with $1<p<\infty$ and $q_0<q<\infty$, we have
$$\left\|\omega\rightarrow N\big(([T_tf](\omega))_{t\geq0},\lambda\big)^{\frac{1}{q}}\right\|_p\lesssim\frac{\|f\|_{L^p(B)}}{\lambda},$$
and for any $K>0$, we also have
$$\mu\big\{\omega\in\Omega|N\big(([T_tf](\omega))_{t\geq0},\lambda\big)>K\big\}\lesssim\frac{\|f\|_{L^p(B)}}{\lambda^pK^{\frac{p}{q}}} .$$
\end{corollary}

The jump estimate is sharp in the sense that the index $q$ can not go below $q_0$. This can be easily seen from the particular case when $\B$ is a Hilbert space. In this case, $q_0=2$ and suppose we can find $q_2<q_0$ such that the jump estimate holds. Then by Lemma 2.1 in \cite{JSW08} (see also \cite{Bou89}), we have corresponding $q$-variational inequality for all $q_2<q$ , which contradicts with the fact the $2$-variational inequality fails (see e.g. \cite{JoWa04}, \cite{Qia98}). For Banach space of martingale cotype $q_0$, similar phenomenon happens. That means we can not expect the jump estimate were true for $q<q_0$. On the other hand, we may expect that the jump estimate would be true when $q=q_0$. All these facts will be proved and appear elsewhere.

\vskip3pt

\noindent \textbf{Acknowledgement.} Guixiang Hong is supported by MINECO: ICMAT Severo Ochoa project SEV-2011-0087 and ERC Grant StG-256997-CZOSQP (EU). Tao Ma is partially supported by NSFC No. 11271292.

\vskip30pt


\begin{thebibliography}{0}
\bibitem{ABG91} N. Asmar, E. Berkson, and T.A. Gillespie, Transference of strong type maximal inequalities by separation-preserving representations, Amer. J. Math. 113: 47-74, 1991.


\bibitem{Bou89} J. Bourgain, Pointwise ergodic theorems for arithmetic sets, Publ. Math. IHES., 69: 5-41, 1989.

\bibitem{CJRW00} J. Campbell, R. Jones, K. Reinhold, and M. Wierdl, Oscillation and Variation for the Hilbert Transform, Duke Math. J. 105: 59-83, 2000.

\bibitem{CJRW03} J.T. Campbell, R.L. Jones, K. Reinhold, and M. Wierdl. Oscillation and variation for singular integrals in higher dimensions,  Trans. Amer. Math. Soc., 355: 2115-2137, 2003.

\bibitem{CoWe77} R. Coifman, G. Weiss, Transference methods in analysis, CBMS 31, Amer. Math. Soc., 1977.


\bibitem{CoLe11} M. Cowling, M. Leinert, Pointwise convergence and semigroups acting on vectro-valued functions, Bull. Aust. Math. Soc., 84:44-48, 2011.

\bibitem{de76} A. de la Torre, A simple proof of the maximal ergodic theorem, Canad. J. Math. 28:1073-1075, 1976.

\bibitem{DMT12} Y. Do, C. Muscalu, and C. Thiele, Variational estimates for paraproducts, Rev. Mat. Iberoam., 28: 857-878, 2012.

\bibitem{DuSc58} N. Dunford, J.T. Schwartz, Linear operators, Part 1, Pure and Applied Mathematics, vol. 7 Interscience Publishers, Inc., New York; Interscience Publisher, Ltd., London 1958 xiv+858 pp.

\bibitem{Fen97} G. Fendler, Dilations of one parameter semigroups of positive contractions on $L_p$-spaces, Canad. J. Math. 49: 736-748, 1997. 

\bibitem{gar-rubio} J. Garc{\'{\i}}a-Cuerva, and J.L. Rubio de Francia, Weighted norm inequalities and related topics, North-Holland Publishing Co., Amsterdam, 1985.

\bibitem{GiTo04} T.A. Gillespie, J.L. Torrea, Dimension free estimates for the oscillation of Riesz transforms, Israel Journal of Mathematics, 141: 125-144, 2004.


\bibitem{Hon} G. Hong, The behaviour of square functions from ergodic theory in $L^\infty$, arXiv:1410.1574.


\bibitem{JKRW98} R. Jones, R. Kaufman, J. Rosenblatt and M. Wierdl, Oscillation in ergodic theory, Ergodic Theory and Dynamical Systems, 18: 889-935, 1998.

\bibitem{JRW03} R. Jones, J.M. Rosenblatt and M. Wierdl, Oscillation in ergodic theory: Higher dimensional results, Israel Journal of Mathematics, 135: 1-27, 2003.

\bibitem{JSW08} R. L. Jones, A. Seeger, and J. Wright. Strong variational and jump inequalities in harmonic analysis, Trans. Amer. Math. Soc. 360: 6711-6742, 2008.

\bibitem{JoWa04} R. L. Jones, G. Wang, Variation inequalities for the Fej\'er and Poisson kernels, Trans. Amer. Math. Soc., 356: 4493-4518, 2004.

\bibitem{JuXu07} M. Junge, Q. Xu, Noncommutative maximal ergodic theorems, J. Amer. Math. Soc., 20: 385-439, 2007.


\bibitem{LeXu2} C. Le Merdy, Q. Xu, Strong q-variation inequalities for analytic semigroups, Ann. Inst. Fourier. 62: 2069-2097, 2012.

\bibitem{Lep76} D. L\'epingle, La variation d'ordre $p$ des semi-martingales, Z. Wahr. Verw. Gebiete 36:295-316.




\bibitem{MTX06} T. Mart\'inez, J. L. Torrea, and Q. Xu,  Vector-valued Littlewood-Paley-Stein theory for semigroups. Adv. Math. 203(2): 430-475, 2006.

\bibitem{Mas} A. Mas. Variation for singular integrals on Lipschitz graphs: $L^p$ and endpoint estimates, Trans. Amer. Math. Soc. To appear.

\bibitem {MaTo12} A. Mas, and X. Tolsa, Variation and oscillation for singular integrals with odd kernel on Lipschitz graphs, Proc. London Math. Soc, 105:49-86, 2012.

\bibitem{MaTo} A. Mas, and X. Tolsa, Variation for the Riesz transform and uniform rectifiability, J. Eur. Math. Soc., To appear.

\bibitem{Mey91} P. Meyer-Nieberg, Banach Lattices, Springer, Berlin-Heidelberg-New York, 1991.

\bibitem{OSTTW12} R. Oberlin, A. Seeger, T. Tao, C. Thiele, and J. Wright,  A variation norm Carleson theorem, J. Eur. Math. Soc., 14: 421-464, 2012.


\bibitem{Pel76} V. Peller, An analogue of J. von Neumann's inequality for the space $L^p$ (Russian), Dokl. Akad. Nauk SSSR 231(3): 539-542, 1976.

\bibitem{Pis92} G. Pisier, Factorization of Operator Valued Analytic Functions, Advance in Mathematics, 93(1): 61-125, 1992.

\bibitem{Pis94} G. Pisier, Complex interpolation and regular operators between Banach spaces, Arch. Math. (Basel) 62(3): 261-269, 1994.

\bibitem{PiXu88} G. Pisier, Q. Xu, The strong $p$-variation of martingales and orthogonal series, Proba. Th. Rel. Fields 77: 497-514, 1988.

\bibitem{Qia98} J. Qian, The $p$-variation of partial sum processes and the empirical process, Ann. of Prob., 26: 1370-1383, 1998.


\bibitem{Ste70} E.M. Stein, Topics in harmonic analysis related to the Littlewood-Paley theory, Ann. Math. Studies, Princeton, University Press, 1970.

\bibitem{Tag09} R.J. Taggart, Pointwise convergence for semigroups in vector-valued L p spaces, Math. Z., 261: 933-949, 2009.

\end{thebibliography}
\end{document}